\documentclass[9pt]{article}%
\usepackage{bbm}
\usepackage{color}
\usepackage{epsfig}
\usepackage{graphics,graphicx,amssymb,amsmath,verbatim}
\usepackage{amssymb}
\usepackage{mathrsfs}
\usepackage{amsfonts}
\usepackage{amsmath}
\usepackage{graphicx}
\usepackage{easybmat}%
\providecommand{\U}[1]{\protect \rule{.1in}{.1in}}
\newtheorem{theorem}{Theorem}

\newtheorem{assumption}{Assumption}

\newtheorem{definition}{Definition}

\newtheorem{lemma}{Lemma}

\newtheorem{proposition}{Proposition}
\newtheorem{remark}{Remark}

\newenvironment{proof}[1][Proof]{\noindent \textbf{#1.} }{\  \rule{0.5em}{0.5em}}

\allowdisplaybreaks[4]
\begin{document}

\title{Stabilization and Consensus of Linear Systems with Multiple Input Delays by
Truncated Pseudo-Predictor Feedback}
\author{Bin Zhou\thanks{Center for Control Theory and Guidance Technology, Harbin
Institute of Technology, Harbin, 150001, China.
Email:\texttt{binzhou@hit.edu.cn, binzhoulee@163.com.}}\quad \quad Shen
Cong\thanks{School of Mechanical \& Electrical Engineering, Heilongjiang
University, Harbin, 150080, P. R. China. Email: \texttt{shen\_tsong@163.com}}}
\date{}
\maketitle

\begin{abstract}
This paper provides an alternative approach referred to as pseudo-predictor
feedback (PPF) for stabilization of linear systems with multiple input delays.
Differently from the traditional predictor feedback which is from the model
reduction appoint of view, the proposed PPF utilizes the idea of prediction by
generalizing the corresponding results for linear systems with a single input
delay to the case of multiple input delays. Since the PPF will generally lead
to distributed controllers, a truncated pseudo-predictor feedback (TPPF)
approach is established instead which gives finite dimensional controllers. It
is shown that the TPPF can compensate arbitrarily large yet bounded delays as
long as the open-loop system is only polynomially unstable. The proposed TPPF
approach is then used to solve the consensus problems for multi-agent systems
characterized by linear systems with multiple input delays. Numerical examples
show the effectiveness of the proposed approach.

\vspace{0.3cm}


\end{abstract}

\section{Introduction}

Stability analysis and stabilization of time-delay systems have received
considerable attention during the past several decades since they are
respectively the foundations of analysis and design for this class of systems
\cite{cz11auto,wssc13tc,wssc11tnn}. For stability analysis, there are tools
built in both frequency domain \cite{gck03book,ok14iet} and time domain
\cite{dlsw09ietcta,lxhz12ijcta,souza14ietcta} and they are equivalently powerful in
a sense. For stabilization of time-delay systems, there are also two kinds of
approaches. The first kind of approaches searches for finite dimensional
controllers which lead to an infinite dimensional closed-loop system whose
stability is generally delay-dependent, namely, its stability cannot
guaranteed for arbitrarily large delay in general. The second kind of
approaches is to design an infinite dimensional controller, which may make the
overall closed-loop system behave as a delay-free linear system and thus its
stability is independent of the delay in general. One of the most efficient
methods within the second category of approaches is the predictor feedback
\cite{artstein82tac,krstic10tac,mo79tac,zhou14book,zld12auto}.

Consider a linear system with a pure input delay%
\begin{equation}
\dot{x}\left(  t\right)  =Ax\left(  t\right)  +B_{0}u\left(  t-\tau
_{0}\right)  , \label{sys00}%
\end{equation}
where $A\in \mathbf{R}^{n\times n}$ and $B_{0}\in \mathbf{R}^{n\times m}$ are
constant matrices and $\tau_{0} \geq0$ is a constant scalar denoting the input
delay. To use the predictor feedback approach to stabilize system
(\ref{sys00}), we can design
\begin{equation}
u\left(  t\right)  =Fx\left(  t+\tau_{0}\right)  , \label{eq80}%
\end{equation}
such that the closed-loop system is delay-free, namely,%
\begin{equation}
\dot{x}\left(  t\right)  =\left(  A+B_{0}F_{0}\right)  x, \label{eq81}%
\end{equation}
which is asymptotically stable if $A+B_{0}F_{0}$ is Hurwitz. However,
controller (\ref{eq80}) is acausal as it involves the future state vector
$x\left(  t+\tau_{0}\right)  .$ There are three prediction methods for making
(\ref{eq80}) be causal:

\begin{itemize}
\item[A.] The traditional Predictor Feedback (PF): by the variation of the
constant formula we predict $x\left(  t+\tau_{0}\right)  $ from the
\emph{open-loop} system (\ref{sys00}) to get
\begin{equation}
u\left(  t\right)  =F_{0}\left(  \mathrm{e}^{A\tau_{0}}x\left(  t\right)
+\int_{t}^{t+\tau_{0}}\mathrm{e}^{A\left(  t+\tau_{0}-s\right)  }B_{0}u\left(
s-\tau_{0}\right)  \mathrm{d}s\right)  , \label{pf}%
\end{equation}
which is implementable. However, this controller is infinite dimensional and
may suffice some implementation problems \cite{vdlr99cdc}.

\item[B.] Truncated Predictor Feedback (TPF): to overcome the shortcoming
existing in the PF (\ref{pf}), under the condition that $\left(  A,B\right)  $
is controllable and all the eigenvalues of $A$ are on the imaginary axis, we
drop the distributed term in it to yield%
\begin{equation}
u\left(  t\right)  =F_{0}\mathrm{e}^{A\tau_{0}}x\left(  t\right)  ,
\label{tpf0}%
\end{equation}
which is a finite dimensional controller and will be referred to as TPF. We
have shown in \cite{zld12auto} that the closed-loop system consisting of
(\ref{sys00}) and (\ref{tpf0}) is asymptotically stable for arbitrarily large
delay $\tau_{0}$ if $F_{0}$ is well designed.

\item[C.] Pseudo-predictor Feedback (PPF): differently from the PF (\ref{pf}),
we predict $x\left(  t+\tau_{0}\right)  $ from the \emph{closed-loop} system
(\ref{eq81}) to get%
\begin{equation}
u\left(  t\right)  =F_{0}\mathrm{e}^{\left(  A+B_{0}F_{0}\right)  \tau_{0}%
}x\left(  t\right)  , \label{ppf0}%
\end{equation}
which is also a finite dimensional controller and will be referred to as PPF.
It is shown in \cite{zhou14auto} that the closed-loop system is asymptotically
stable if and only if the following integral delay system is%
\begin{equation}
\rho_{0}\left(  t\right)  =-F_{0}\int_{t-\tau_{0}}^{t}\mathrm{e}^{\left(
A+B_{0}F_{0}\right)  \left(  t-s\right)  }B_{0}\rho_{0}\left(  s\right)
\mathrm{d}s, \label{ids0}%
\end{equation}
which has been studied in \cite{lzl13auto}. Moreover, similar to the TPF
(\ref{tpf0}), if $\left(  A,B\right)  $ is controllable and all the
eigenvalues of $A$ are on the imaginary axis, the closed-loop system
consisting of (\ref{sys00}) and (\ref{ppf0}) is asymptotically stable for
arbitrarily large yet bounded delay $\tau_{0}$ if $F_{0}$ is well designed
\cite{zhou14auto}.
\end{itemize}

The aim of the present paper is to generalize the PPF approach for the linear
system (\ref{sys00}) with a single input delay to the following linear system
with multiple input delays%
\begin{equation}
\dot{x}\left(  t\right)  =Ax\left(  t\right)  +\sum \limits_{i=1}^{p}%
B_{i}u\left(  t-\tau_{i}\right)  , \label{general}%
\end{equation}
where $A\in \mathbf{R}^{n\times n}$ and $B_{i}\in \mathbf{R}^{n\times
m},i=1,2,\ldots,p,$ are constant matrices, and $\tau_{i},i=1,2,\ldots,p$, are
constant scalars. The main contribution and significance of this paper are
listed below:

\begin{enumerate}
\item Differently from the PPF design for linear system (\ref{sys00}) with a
single input delay, the PPF design for the linear system (\ref{general}) with
multiple input delays cannot be accomplished by using a virtual closed-loop
system to predict the future state since the traditional predictor feedback
(also known as model reduction) for system (\ref{general}) does not use the
future state at all. Hence we have to build an alternative traditional
predictor feedback for system (\ref{general}) based on which a reasonable PPF
can be designed. It turns out that the designed PPF can indeed stabilize
system (\ref{general}) as long as an integral delay system in the form of
(\ref{ids0}) is stable, indicating that our PPF design for system
(\ref{general}) is indeed a generalization of the PPF design for system
(\ref{sys00}) to the case of multiple input delays.

\item However, differently from the PPF for system (\ref{sys00}) with a single
input delay, the PPF for system (\ref{general}) with multiple input delays
results in an infinite dimensional controller which can only be implemented by
numerical approximation and may lead to instability to the overall closed-loop
system. To overcome this shortcoming, under the condition that $\left(
A,B\right)  $ is controllable and all the eigenvalues of $A$ are on the
imaginary axis, we provide a truncated pseudo-predictor feedback (TPPF)
approach which leads to a finite dimensional controller that can stabilize
system (\ref{general}) for arbitrarily large yet bounded delays if the
feedback gain is well designed. A numerical comparison of the proposed TPPF
with our early proposed TPF shows that the TPPF can generally lead to much
better control performances.

\item To show the effectiveness of the TPPF design for linear system
(\ref{general}) with multiple input delays, we use it to solve the consensus
problem for multi-agent systems modeled by general high-order linear systems
with multiple input delays. It is shown that as long as the open dynamics is
not exponentially stable and the communication network contains a spanning
tree, the TPPF based protocols can ensure the consensus of the concerned
multi-agent systems for arbitrarily large yet bounded delays.
\end{enumerate}

The reminder of this paper is organized as follows. After formulating the
problem in Section \ref{sec1}, the PPF design and TPPF design are respectively
carried out in Sections \ref{sec2} and \ref{sec3}. Application of the TPPF
design to the consensus of multi-agent systems is then discussed in Section
\ref{sec4}. Two numerical examples are worked out in Section \ref{sec5} to
show the effectiveness of the proposed approaches and Section \ref{sec6}%
\ concludes the paper.

\textbf{Notation:} The notation used in this paper is fairly standard. The
symbol $\lambda \left(  A\right)  $ denotes the eigenvalue set of the square
matrix $A.$ For a positive scalar $\tau,$ let $\mathscr{C}_{n,\tau
}=\mathscr{C}\left(  \left[  -\tau,0\right]  ,\mathbf{R}^{n}\right)  $ denote
the Banach space of continuous vector functions mapping the interval $\left[
-\tau,0\right]  $ into $\mathbf{R}^{n}$ with the topology of uniform
convergence, and let $x_{t}\in \mathscr{C}_{n,\tau}$ denote the restriction of
$x\left(  t\right)  $ to the interval $\left[  t-\tau,t\right]  $ translated
to $\left[  -\tau,0\right]  ,$ that is, $x_{t}\left(  \theta \right)  =x\left(
t+\theta \right)  ,\theta \in \left[  -\tau,0\right]  $. For two matrices $A$ and
$B,$ we use $A\otimes B$ to denote their Kronecker product. For two integers
$p$ and $q$ with $p\leq q,$ the symbol $\mathbf{I}\left[  p,q\right]  $ refers
to the set $\left \{  p,p+1,\cdots,q\right \}  .$ Finally, for a complex number
$s,$ we use $\operatorname{Re}\{s\}$ and $\left \vert s\right \vert $ to denote
respectively its real part and model.

\section{\label{sec1}System Description and Problem Formulation}

Consider the following linear system with three input delays
\begin{equation}
\dot{x}\left(  t\right)  =Ax\left(  t\right)  +B_{1}u\left(  t-\tau
_{1}\right)  +B_{2}u\left(  t-\tau_{2}\right)  +B_{3}u\left(  t-\tau
_{3}\right)  , \label{sys1}%
\end{equation}
where $A\in \mathbf{R}^{n\times n}$ and $B_{i}\in \mathbf{R}^{n\times
m},i=1,2,3,$ are constant matrices, and $\tau_{i},i=1,2,3,$ are constant
scalars and such that
\begin{equation}
0\leq \tau_{1}\leq \tau_{2}\leq \tau_{3}<\infty. \label{delay1}%
\end{equation}
For future use, we denote%
\begin{equation}
\tau=\tau_{1}+\tau_{2}+\tau_{3}. \label{tau}%
\end{equation}
It is well-known that system (\ref{sys1}) is stabilizable if and only if the
following assumption is satisfied \cite{artstein82tac}:

\begin{assumption}
\label{assup1}The matrix pair $\left(  A,B\right)  $ is stabilizable, where%
\begin{equation}
B=\mathrm{e}^{-A\tau_{3}}B_{3}+\mathrm{e}^{-A\tau_{2}}B_{2}+\mathrm{e}%
^{-A\tau_{1}}B_{1}, \label{eqb}%
\end{equation}
namely, there exists a matrix $F\in \mathbf{R}^{m\times n}$ such that $A+BF$ is Hurwitz.
\end{assumption}

In this paper, we are interested in the design of state feedback stabilizing
controller $u\left(  t\right)  $ for system (\ref{sys1}) under Assumption
\ref{assup1}. One of the existing approaches for solving this problem is the
so-called model reduction approach, which is described as follows. Let%
\begin{equation}
z\left(  t\right)  =x\left(  t\right)  +%
%TCIMACRO{\dsum \limits_{i=1}^{3}}%
%BeginExpansion
{\displaystyle \sum \limits_{i=1}^{3}}
%EndExpansion
\int_{t-\tau_{i}}^{t}\mathrm{e}^{A\left(  t-\tau_{i}-s\right)  }B_{i}u\left(
s\right)  \mathrm{d}s. \label{zt}%
\end{equation}
Then direct computation gives (see, for example, \cite{artstein82tac} and
\cite{zld10siamjco})%
\begin{equation}
\dot{z}\left(  t\right)  =Az\left(  t\right)  +Bu\left(  t\right)  ,
\label{eq10}%
\end{equation}
which is a delay-free linear system. Then, under Assumption \ref{assup1}, a
stabilizing controller can be designed as%
\begin{align}
u\left(  t\right)   &  =Fz\left(  t\right) \nonumber \\
&  =F\left(  x\left(  t\right)  +%
%TCIMACRO{\dsum \limits_{i=1}^{3}}%
%BeginExpansion
{\displaystyle \sum \limits_{i=1}^{3}}
%EndExpansion
\int_{t-\tau_{i}}^{t}\mathrm{e}^{A\left(  t-\tau_{i}-s\right)  }B_{i}u\left(
s\right)  \mathrm{d}s\right)  , \label{eq11}%
\end{align}
which is implementable.

However, as the model reduction controller (\ref{eq11}) involves the
integration of $u\left(  t+s\right)  ,s\leq0,$ it must be implemented via
numerical approximation, which may lead to some stability problems (see, for
example, \cite{vdlr99cdc}). To overcome this problem, we proposed in
\cite{zhou14book} and \cite{zld10siamjco} a truncated version of the controller
(\ref{eq11}), namely,%
\begin{equation}
u\left(  t\right)  =Fx\left(  t\right)  , \label{tpf}%
\end{equation}
which is a finite dimensional controller and is referred to as truncated
predictor feedback (TPF) \cite{zhou14book}, under the following assumption:

\begin{assumption}
\label{ass2}The matrix pair $(A,B)\in \left(  \mathbf{R}^{n\times n}%
\times \mathbf{R}^{n\times m}\right)  $ is controllable with all the
eigenvalues of $A$ being on the imaginary axis.
\end{assumption}

Under this assumption, the following result has been proven in
\cite{zhou14book} and \cite{zld10siamjco}.

\begin{lemma}
\label{lm0}Let $\left(  A,B\right)  $ satisfy Assumption \ref{ass2} and
$F=-B^{\mathrm{T}}P$ where $P$ is the unique positive definite solution to the
following parametric algebraic Riccati equation (ARE) (see Lemma \ref{tppflm1}
given later)%
\begin{equation}
A^{\mathrm{T}}P+PA-PBB^{\mathrm{T}}P=-\gamma P. \label{are}%
\end{equation}
Then there exists a $\gamma^{\ast}>0$ such that the TPF (\ref{tpf}) stabilizes
system (\ref{sys1}) for all $\gamma \in \left(  0,\gamma^{\ast}\right)  .$
\end{lemma}

In this paper, we will establish an alternative approach for stabilizing
system (\ref{sys0}) from the prediction point of view and will also build a
class of finite dimensional stabilizing controllers under Assumption
\ref{ass2}. This new approach provides an option for the designer since it may
possess much better control performance than the TPF (\ref{tpf}).

\begin{remark}
For easy presentation, in this paper we have assumed that there are only three
delays in the inputs. However, the reported results can be directly
generalized to the more general system (\ref{general}).
\end{remark}

\section{\label{sec2}Stabilization by Pseudo-Predictor Feedback}

\subsection{A Recursive Derivation of the Model Reduction Approach}

We first consider the following new state vector%
\begin{equation}
y_{3}\left(  t\right)  =\mathrm{e}^{A\tau_{3}}x\left(  t\right)  +\int
_{t}^{t+\tau_{3}}\mathrm{e}^{A\left(  t+\tau_{3}-s\right)  }B_{3}u\left(
s-\tau_{3}\right)  \mathrm{d}s \label{y3}%
\end{equation}
by which and some computation we can rewrite system (\ref{sys1}) as%
\begin{equation}
\dot{y}_{3}\left(  t\right)  =Ay_{3}\left(  t\right)  +B_{3}u\left(  t\right)
+\mathrm{e}^{A\tau_{3}}B_{1}u\left(  t-\tau_{1}\right)  +\mathrm{e}^{A\tau
_{3}}B_{2}u\left(  t-\tau_{2}\right)  . \label{eq2}%
\end{equation}
Now choose another new state vector $y_{2}\left(  t\right)  $ as%
\begin{equation}
y_{2}\left(  t\right)  =\mathrm{e}^{A\tau_{2}}y_{3}\left(  t\right)  +\int
_{t}^{t+\tau_{2}}\mathrm{e}^{A\left(  t+\tau_{2}-s\right)  }\left(
\mathrm{e}^{A\tau_{3}}B_{2}u\left(  s-\tau_{2}\right)  \right)  \mathrm{d}s,
\label{y2}%
\end{equation}
by which system (\ref{eq2}) is transformed into%
\begin{equation}
\dot{y}_{2}\left(  t\right)  =Ay_{2}\left(  t\right)  +\left(  \mathrm{e}%
^{A\tau_{2}}B_{3}+\mathrm{e}^{A\tau_{3}}B_{2}\right)  u\left(  t\right)
+\mathrm{e}^{A\left(  \tau_{2}+\tau_{3}\right)  }B_{1}u\left(  t-\tau
_{1}\right)  . \label{eq3}%
\end{equation}
Finally, we choose the following new state vector $y_{1}\left(  t\right)  :$%
\begin{equation}
y_{1}\left(  t\right)  =\mathrm{e}^{A\tau_{1}}y_{2}\left(  t\right)  +\int
_{t}^{t+\tau_{1}}\mathrm{e}^{A\left(  t+\tau_{1}-s\right)  }\left(
\mathrm{e}^{A\left(  \tau_{2}+\tau_{3}\right)  }B_{1}u\left(  s-\tau
_{1}\right)  \right)  \mathrm{d}s, \label{y1}%
\end{equation}
and it follows from (\ref{eq3}) that%
\begin{equation}
\dot{y}_{1}\left(  t\right)  =Ay_{1}\left(  t\right)  +\mathrm{e}^{A\tau
}\left(  \mathrm{e}^{-A\tau_{3}}B_{3}+\mathrm{e}^{-A\tau_{2}}B_{2}%
+\mathrm{e}^{-A\tau_{1}}B_{1}\right)  u\left(  t\right)  . \label{eq4}%
\end{equation}
Now we consider a state transformation%
\begin{equation}
y\left(  t\right)  =\mathrm{e}^{-A\tau}y_{1}\left(  t\right)  , \label{yt}%
\end{equation}
by which system (\ref{eq4}) can be expressed as%
\begin{equation}
\dot{y}\left(  t\right)  =Ay\left(  t\right)  +Bu\left(  t\right)  ,
\label{sys0}%
\end{equation}
which is also in the form of (\ref{eq10}).

Since (\ref{sys0}) is a delay-free linear system, under Assumption
\ref{assup1}, if we choose%
\begin{equation}
u\left(  t\right)  =Fy\left(  t\right)  , \label{ut}%
\end{equation}
then the closed-loop system is given by%
\begin{equation}
\dot{y}\left(  t\right)  =\left(  A+BF\right)  y\left(  t\right)  ,
\label{closed}%
\end{equation}
which is asymptotically stable if $A+BF$ is Hurwitz.

To write the controller (\ref{ut}) in terms of the original state vector
$x\left(  t\right)  ,$ we present the following lemma.

\begin{lemma}
\label{lm1} Let $y\left(  t\right)  $ be defined through (\ref{y3}),
(\ref{y2}), (\ref{y1}) and (\ref{yt}) and $z\left(  t\right)  $ be defined in
(\ref{zt}). Then%
\begin{equation}
y\left(  t\right)  =z\left(  t\right)  . \label{yz}%
\end{equation}

\end{lemma}

Hence the controller (\ref{ut}) is exactly the model reduction based
controller (\ref{eq11}). However, the aim of the recursive design of the model
reduction described in the above is not to yield the controller (\ref{eq11})
but is to link the new state vector $z\left(  t\right)  $ or $y\left(
t\right)  $ with the future state vector $x\left(  t+\tau \right)  $ and the
future input $u\left(  t+s\right)  ,s>0$, as stated in the following lemma
whose proof is provided in the appendix for clarity.

\begin{lemma}
\label{lm2} Let $y\left(  t\right)  $ be defined through (\ref{y3}),
(\ref{y2}), (\ref{y1}) and (\ref{yt}). Then%
\begin{equation}
y\left(  t\right)  =\mathrm{e}^{-A\tau}x\left(  t+\tau \right)  -%
%TCIMACRO{\dsum \limits_{i=1}^{3}}%
%BeginExpansion
{\displaystyle \sum \limits_{i=1}^{3}}
%EndExpansion
\int_{t}^{t+\tau-\tau_{i}}\mathrm{e}^{A\left(  t-s-\tau_{i}\right)  }%
B_{i}u\left(  s\right)  \mathrm{d}s. \label{yx}%
\end{equation}

\end{lemma}

The connection revealed in the above lemma helps us to derive a predictor
based approach for stabilization of system (\ref{sys0}) in the next subsection.

\subsection{Pseudo-Predictor Feedback}

Since the right-hand side of (\ref{yx}) is acausal, we consider%
\begin{align}
y\left(  t-\tau \right)   &  =\mathrm{e}^{-A\tau}x\left(  t\right)  -%
%TCIMACRO{\dsum \limits_{i=1}^{3}}%
%BeginExpansion
{\displaystyle \sum \limits_{i=1}^{3}}
%EndExpansion
\int_{t-\tau}^{t-\tau_{i}}\mathrm{e}^{A\left(  t-\tau-s-\tau_{i}\right)
}B_{i}u\left(  s\right)  \mathrm{d}s\nonumber \\
&  =\mathrm{e}^{-A\tau}\left(  x\left(  t\right)  -%
%TCIMACRO{\dsum \limits_{i=1}^{3}}%
%BeginExpansion
{\displaystyle \sum \limits_{i=1}^{3}}
%EndExpansion
\int_{t-\tau}^{t-\tau_{i}}\mathrm{e}^{A\left(  t-s-\tau_{i}\right)  }%
B_{i}u\left(  s\right)  \mathrm{d}s\right) \nonumber \\
&  \triangleq \mathrm{e}^{-A\tau}\left(  x\left(  t\right)  -\varphi \left(
t\right)  \right)  , \label{eq6}%
\end{align}
and use the closed-loop system model (\ref{closed}) to yield%
\begin{align}
u\left(  t\right)   &  =Fy\left(  t\right) \nonumber \\
&  =F\mathrm{e}^{\left(  A+BF\right)  \tau}y\left(  t-\tau \right) \nonumber \\
&  =F\mathrm{e}^{\left(  A+BF\right)  \tau}\mathrm{e}^{-A\tau}\left(  x\left(
t\right)  -\varphi \left(  t\right)  \right)  , \label{ppf}%
\end{align}
which will be referred to as pseudo-predictor feedback (PPF) since we have
used the virtual closed-loop system (\ref{closed}) to predict $y\left(
t\right)  $ from $y\left(  t-\tau \right)  $ \cite{zhou14auto}. The controller
(\ref{ppf}) is now causal and is thus implementable.

Regarding the stability of the closed-loop system under the PPF (\ref{ppf}),
we have the following result.

\begin{proposition}
\label{pp0}The closed-loop system consisting of (\ref{sys00}) and the PPF
(\ref{ppf}) is asymptotically stable if and only if the following integral
delay system is%
\begin{equation}
\rho \left(  t\right)  =-F\int_{t-\tau}^{t}\mathrm{e}^{\left(  A+BF\right)
\left(  t-s\right)  }B\rho \left(  s\right)  \mathrm{d}s. \label{ids}%
\end{equation}

\end{proposition}

\begin{proof}
By (\ref{ppf}) and (\ref{sys0}), the closed-lop system is given by%
\begin{align}
\dot{y}\left(  t\right)   &  =Ay\left(  t\right)  +BF\mathrm{e}^{\left(
A+BF\right)  \tau}\mathrm{e}^{-A\tau}\left(  x\left(  t\right)  -\varphi
\left(  t\right)  \right) \nonumber \\
&  =Ay\left(  t\right)  +BF\mathrm{e}^{\left(  A+BF\right)  \tau}y\left(
t-\tau \right) \nonumber \\
&  =\left(  A+BF\right)  y\left(  t\right)  +B\rho \left(  t\right)  ,
\label{eq7}%
\end{align}
where%
\begin{equation}
\rho \left(  t\right)  =F\left(  \mathrm{e}^{\left(  A+BF\right)  \tau}y\left(
t-\tau \right)  -y\left(  t\right)  \right)  . \label{rt}%
\end{equation}
On the other hand, it follows from (\ref{eq7}) that%
\begin{equation}
y\left(  t\right)  =\mathrm{e}^{\left(  A+BF\right)  \tau}y\left(
t-\tau \right)  +\int_{t-\tau}^{t}\mathrm{e}^{\left(  A+BF\right)  \left(
t-s\right)  }B\rho \left(  s\right)  \mathrm{d}s, \label{eq18}%
\end{equation}
substituting of which into (\ref{eq7}) gives the integral delay system
(\ref{ids}). Hence if (\ref{ids}) is asymptotically stable, it follows from
(\ref{eq7}) that the closed-loop system is stable. The converse can be shown
similarly as the proof of Theorem 1 in \cite{zhou14auto} and is omitted for brevity.
\end{proof}

\begin{remark}
The PPF (\ref{ppf}) will reduce to the PPF in \cite{zhou14auto} if there is
only one delay in the system, namely, $B_{1}=B_{2}=0$ and $\tau_{1}=\tau
_{2}=0.$ In fact, in this case, (\ref{ppf}) reduces to%
\begin{align}
u\left(  t\right)   &  =F\mathrm{e}^{\left(  A+BF\right)  \tau_{3}}%
\mathrm{e}^{-A\tau_{3}}x\left(  t\right) \nonumber \\
&  =F\mathrm{e}^{\left(  A+\mathrm{e}^{-A\tau_{3}}B_{3}F\right)  \tau_{3}%
}\mathrm{e}^{-A\tau_{3}}x\left(  t\right) \nonumber \\
&  =F\mathrm{e}^{\mathrm{e}^{-A\tau_{3}}\left[  \left(  A+B_{3}F\mathrm{e}%
^{-A\tau_{3}}\right)  \tau_{3}\right]  \mathrm{e}^{A\tau_{3}}}\mathrm{e}%
^{-A\tau_{3}}x\left(  t\right) \nonumber \\
&  =F\mathrm{e}^{-A\tau_{3}}\mathrm{e}^{\left(  A+B_{3}F\mathrm{e}^{-A\tau
_{3}}\right)  \tau_{3}}\mathrm{e}^{A\tau_{3}}\mathrm{e}^{-A\tau_{3}}x\left(
t\right) \nonumber \\
&  =F_{3}\mathrm{e}^{\left(  A+B_{3}F_{3}\right)  \tau_{3}}x\left(  t\right)
, \label{ppf1}%
\end{align}
where $F_{3}=F\mathrm{e}^{-A\tau_{3}}$ and is such that $A+B_{3}F_{3}$ is
Hurwitz. This controller (\ref{ppf1}) is exactly in the form of the PPF
(\ref{ppf0}) established in \cite{zhou14auto}. On the other hand, the integral
delay system (\ref{ids}) is exactly in the form of (\ref{ids0}). Hence
Proposition \ref{pp0}\ can be naturally considered a generalization of the
results for the PPF established in \cite{zhou14auto} to the case of multiple
input delays.
\end{remark}

Differently from the PPF (\ref{ppf0}) for the linear system (\ref{sys00}) with
a single input delay, the PPF (\ref{ppf}) for system (\ref{sys1}) with
multiple input delays also involves the integration of $u\left(  t\right)  $
in the past time and is thus an infinite dimensional controller which may
suffice some implementation problems. To overcome this shortcoming, we will
establish a truncated version of this class of controller in the next section.

\section{\label{sec3}Truncated Pseudo-Predictor Feedback}

\subsection{Derivation of the Truncated Pseudo-Predictor Feedback}

Similarly to the development in \cite{zld12auto}, let the feedback gain $F$ be
parameterized as $F=F\left(  \gamma \right)  :\mathbf{R}\rightarrow
\mathbf{R}^{m\times n}$ and such that%
\begin{equation}
\lim_{\gamma \downarrow0}\frac{1}{\gamma}\left \Vert F\left(  \gamma \right)
\right \Vert <\infty,\label{eq50}%
\end{equation}
then $u\left(  t\right)  $ is also \textquotedblleft of order
1\textquotedblright \ with respect to $\gamma$. Consequently, the vector
$\varphi \left(  t\right)  $ is \textquotedblleft of order 1\textquotedblright%
\ with respect to $\gamma.$ As a result, $F\mathrm{e}^{\left(  A+BF\right)
\tau}\mathrm{e}^{-A\tau}\varphi \left(  t\right)  $ is at least
\textquotedblleft of order 2\textquotedblright \ with respect to $\gamma$.
Hence, the term $F\mathrm{e}^{\left(  A+BF\right)  \tau}\mathrm{e}^{-A\tau
}\varphi \left(  t\right)  $ is dominated by the term $F\mathrm{e}^{\left(
A+BF\right)  \tau}\mathrm{e}^{-A\tau}x\left(  t\right)  $ and thus might be
safely neglected in $u\left(  t\right)  $ when $\gamma$ is sufficiently small
\cite{zld12auto}. As a result, the PPF (\ref{ppf}) can be truncated as%
\begin{equation}
u\left(  t\right)  =F\mathrm{e}^{\left(  A+BF\right)  \tau}\mathrm{e}^{-A\tau
}x\left(  t\right)  ,\label{tppf}%
\end{equation}
which we refer to as the \textquotedblleft truncated pseudo-predictor feedback
(TPPF).\textquotedblright \ The main advantage of the TPPF (\ref{tppf}) over
the PPF (\ref{ppf}) is that the former one is finite dimensional and is easy
to implement.

According to the discussion in \cite{zld12auto}, to ensure that there exists a
stabilizing $F$ satisfying (\ref{eq50}), we should assume that system
(\ref{sys1}) satisfies Assumption \ref{ass2}. In this case, we have the
following lemma which is helpful for proving our main results.

\begin{lemma}
\label{tppflm1}\cite{zhou14auto} Assume that $\left(  A,B\right)  $ satisfies
Assumption \ref{ass2}. Then the ARE (\ref{are}) admits a unique positive
definite solution $P\left(  \gamma \right)  =W^{-1}\left(  \gamma \right)  $ for
all $\gamma>0$, where $W\left(  \gamma \right)  $ is the unique positive
definite solution to the Lyapunov equation
\begin{equation}
W\left(  A+\frac{\gamma}{2}I_{n}\right)  ^{\mathrm{T}}+\left(  A+\frac{\gamma
}{2}I_{n}\right)  W=BB^{\mathrm{T}}. \label{ppfeq34}%
\end{equation}
Moreover, $A+BF$ is Hurwitz, where $F=-B^{\mathrm{T}}P$, $\mathrm{tr}\left(
B^{\mathrm{T}}PB\right)  =n\gamma,PBB^{\mathrm{T}}P\leq n\gamma P\left(
\gamma \right)  ,\lim_{\gamma \downarrow0}P\left(  \gamma \right)  =0,$ and
\begin{equation}
P\mathrm{e}^{\left(  A+BF\right)  }=\mathrm{e}^{-\left(  A+\gamma
I_{n}\right)  ^{\mathrm{T}}}P. \label{eq29}%
\end{equation}

\end{lemma}

In the next subsection we will show that the TPPF (\ref{tppf}) can indeed
stabilize system (\ref{sys1}).

\subsection{Stability of the Closed-Loop System}

The aim of this subsection is to prove the following result.

\begin{theorem}
\label{th1}Let $F=-B^{\mathrm{T}}P$ where $P$ is the unique positive definite
solution to the parametric ARE (\ref{are}). Then there exits a scalar
$\gamma^{\ast}>0$ such that the closed-loop system consisting of (\ref{sys1})
and the TPPF (\ref{tppf}) is asymptotically stable for all $\gamma \in \left(
0,\gamma^{\ast}\right)  .$
\end{theorem}

\begin{proof}
Consider the closed-loop system consisting of (\ref{sys0}) and (\ref{tppf})
given by%
\begin{align}
\dot{y}\left(  t\right)   &  =Ay\left(  t\right)  +BF\mathrm{e}^{\left(
A+BF\right)  \tau}\mathrm{e}^{-A\tau}x\left(  t\right) \nonumber \\
&  =Ay\left(  t\right)  +BF\mathrm{e}^{\left(  A+BF\right)  \tau}\left(
y\left(  t-\tau \right)  +\mathrm{e}^{-A\tau}\varphi \left(  t\right)  \right)
\nonumber \\
&  =Ay\left(  t\right)  +BF\mathrm{e}^{\left(  A+BF\right)  \tau}y\left(
t-\tau \right)  +BF\mathrm{e}^{\left(  A+BF\right)  \tau}\mathrm{e}^{-A\tau
}\varphi \left(  t\right) \nonumber \\
&  =\left(  A+BF\right)  y\left(  t\right)  +B\rho \left(  t\right)  ,
\label{eq21}%
\end{align}
where%
\begin{equation}
\left \{
\begin{array}
[c]{l}%
\rho \left(  t\right)  =F\left(  \mathrm{e}^{\left(  A+BF\right)  \tau}y\left(
t-\tau \right)  -y\left(  t\right)  \right)  +\phi \left(  t\right)  ,\\
\phi \left(  t\right)  =F\mathrm{e}^{\left(  A+BF\right)  \tau}\mathrm{e}%
^{-A\tau}\varphi \left(  t\right)  .
\end{array}
\right.  \label{eq22}%
\end{equation}
By using the variation of constant formula, we get from (\ref{eq21}) that%
\begin{equation}
y\left(  t\right)  =\mathrm{e}^{\left(  A+BF\right)  \tau}y\left(
t-\tau \right)  +\int_{t-\tau}^{t}\mathrm{e}^{\left(  A+BF\right)  \left(
t-s\right)  }B\rho \left(  s\right)  \mathrm{d}s, \label{eq24}%
\end{equation}
by which $\rho \left(  t\right)  $ can be simplified as%
\begin{equation}
\rho \left(  t\right)  =-F\int_{t-\tau}^{t}\mathrm{e}^{\left(  A+BF\right)
\left(  t-s\right)  }B\rho \left(  s\right)  \mathrm{d}s+\phi \left(  t\right)
. \label{eq23}%
\end{equation}
For later use, we notice from the definition of $\varphi \left(  t\right)  $ in
(\ref{eq6}) that%
\begin{equation}
\left \{
\begin{array}
[c]{l}%
\phi \left(  t\right)  =F\mathrm{e}^{\left(  A+BF\right)  \tau}\mathrm{e}%
^{-A\tau}%
%TCIMACRO{\dsum \limits_{i=1}^{3}}%
%BeginExpansion
{\displaystyle \sum \limits_{i=1}^{3}}
%EndExpansion
\varphi_{i}\left(  t\right)  ,\\
\varphi_{i}\left(  t\right)  =\int_{t-\tau}^{t-\tau_{i}}\mathrm{e}^{A\left(
t-s-\tau_{i}\right)  }B_{i}u\left(  s\right)  \mathrm{d}s,\;i=1,2,3.
\end{array}
\right.  \label{eq25}%
\end{equation}
Moreover, by equation (\ref{eq6}) we can write $u\left(  t\right)  $ as%
\begin{align}
u\left(  t\right)   &  =F\mathrm{e}^{\left(  A+BF\right)  \tau}\left(
y\left(  t-\tau \right)  +\mathrm{e}^{-A\tau}\varphi \left(  t\right)  \right)
\nonumber \\
&  =F\mathrm{e}^{\left(  A+BF\right)  \tau}y\left(  t-\tau \right)
+\phi \left(  t\right)  . \label{eq26}%
\end{align}

Consider a Lyapunov function $V_{1}\left(  y\left(  t\right)  \right)
=y^{\mathrm{T}}\left(  t\right)  Py\left(  t\right)  $ whose time-derivative
along system (\ref{eq21}) is given by%
\begin{align}
\dot{V}_{1}\left(  y\left(  t\right)  \right)   &  =y^{\mathrm{T}}\left(
t\right)  \left(  \left(  A+BF\right)  ^{\mathrm{T}}P+P\left(  A+BF\right)
\right)  y\left(  t\right)  +2y^{\mathrm{T}}\left(  t\right)  PB\rho \left(
t\right) \nonumber \\
&  =-\gamma y^{\mathrm{T}}\left(  t\right)  Py\left(  t\right)  -y^{\mathrm{T}%
}\left(  t\right)  PBB^{\mathrm{T}}Py\left(  t\right)  +2y^{\mathrm{T}}\left(
t\right)  PB\rho \left(  t\right) \nonumber \\
&  \leq-\gamma y^{\mathrm{T}}\left(  t\right)  Py\left(  t\right)  +\left \Vert
\rho \left(  t\right)  \right \Vert ^{2}, \label{eq47}%
\end{align}
where we have used Lemma \ref{tppflm1}. By (\ref{eq23}) and the Jensen
inequality we can derive%
\begin{align}
\left \Vert \rho \left(  t\right)  \right \Vert ^{2}  &  \leq2\left(
\int_{t-\tau}^{t}\mathrm{e}^{\left(  A+BF\right)  \left(  t-s\right)  }%
B\rho \left(  s\right)  \mathrm{d}s\right)  ^{\mathrm{T}}F^{\mathrm{T}}%
F\int_{t-\tau}^{t}\mathrm{e}^{\left(  A+BF\right)  \left(  t-s\right)  }%
B\rho \left(  s\right)  \mathrm{d}s+2\left \Vert \phi \left(  t\right)
\right \Vert ^{2}\nonumber \\
&  \leq2n\gamma \left \Vert P\right \Vert \tau \int_{t-\tau}^{t}\rho^{\mathrm{T}%
}\left(  s\right)  B^{\mathrm{T}}\mathrm{e}^{\left(  A+BF\right)
^{\mathrm{T}}\left(  t-s\right)  }\mathrm{e}^{\left(  A+BF\right)  \left(
t-s\right)  }B\rho \left(  s\right)  \mathrm{d}s+2\left \Vert \phi \left(
t\right)  \right \Vert ^{2}\nonumber \\
&  \leq2n\gamma \left \Vert P\right \Vert \tau \left \Vert B\right \Vert
^{2}\mathrm{e}^{2\left(  \left \Vert A\right \Vert +\left \Vert B\right \Vert
^{2}\left \Vert P\right \Vert \right)  \tau}\int_{t-\tau}^{t}\left \Vert
\rho \left(  s\right)  \right \Vert ^{2}\mathrm{d}s+2\left \Vert \phi \left(
t\right)  \right \Vert ^{2}\nonumber \\
&  =c_{1}\gamma \left \Vert P\right \Vert \int_{t-\tau}^{t}\left \Vert \rho \left(
s\right)  \right \Vert ^{2}\mathrm{d}s+2\left \Vert \phi \left(  t\right)
\right \Vert ^{2}, \label{eq27}%
\end{align}
where
\begin{equation}
c_{1}=c_{1}\left(  \gamma \right)  =2n\tau \left \Vert B\right \Vert
^{2}\mathrm{e}^{2(\left \Vert A\right \Vert +\left \Vert B\right \Vert
^{2}\left \Vert P\right \Vert )\tau}. \label{c1}%
\end{equation}
Let $\gamma_{1}^{\ast}>0$ be such that%
\begin{equation}
c_{1}\gamma \left \Vert P\right \Vert \tau+\sqrt{\left \Vert P\right \Vert }%
\leq1,\; \forall \gamma \in \left(  0,\gamma_{1}^{\ast}\right)  . \label{eq46}%
\end{equation}
Then the time-derivative of the second Lyapunov functional%
\begin{equation}
V_{2}\left(  \rho \left(  t\right)  \right)  =c_{1}\gamma \left \Vert
P\right \Vert \int_{-\tau}^{0}\int_{t-s}^{t}\left \Vert \rho \left(
\sigma \right)  \right \Vert ^{2}\mathrm{d}\sigma \mathrm{d}s \label{eq45}%
\end{equation}
along (\ref{eq27}) satisfies%
\begin{align}
\dot{V}_{2}\left(  \rho \left(  t\right)  \right)   &  =c_{1}\gamma \left \Vert
P\right \Vert \tau \left \Vert \rho \left(  t\right)  \right \Vert ^{2}-c_{1}%
\gamma \left \Vert P\right \Vert \int_{t-\tau}^{t}\left \Vert \rho \left(
\sigma \right)  \right \Vert ^{2}\mathrm{d}\sigma \nonumber \\
&  \leq c_{1}\gamma \left \Vert P\right \Vert \tau \left \Vert \rho \left(
t\right)  \right \Vert ^{2}+2\left \Vert \phi \left(  t\right)  \right \Vert
^{2}-\left \Vert \rho \left(  t\right)  \right \Vert ^{2}\nonumber \\
&  \leq2\left \Vert \phi \left(  t\right)  \right \Vert ^{2}-\sqrt{\left \Vert
P\right \Vert }\left \Vert \rho \left(  t\right)  \right \Vert ^{2}. \label{eq44}%
\end{align}

Now by using (\ref{eq25}) we can obtain%
\begin{align}
\left \Vert \phi \left(  t\right)  \right \Vert ^{2} &  =\varphi^{\mathrm{T}%
}\left(  t\right)  \mathrm{e}^{-A^{\mathrm{T}}\tau}\mathrm{e}^{\left(
A+BF\right)  ^{\mathrm{T}}\tau}F^{\mathrm{T}}F\mathrm{e}^{\left(  A+BF\right)
\tau}\mathrm{e}^{-A\tau}\varphi \left(  t\right)  \nonumber \\
&  \leq n\gamma \left \Vert P\right \Vert \mathrm{e}^{2\left(  \left \Vert
A\right \Vert +\left \Vert B\right \Vert ^{2}\left \Vert P\right \Vert \right)
\tau}\mathrm{e}^{2\left \Vert A\right \Vert \tau}\varphi^{\mathrm{T}}\left(
t\right)  \varphi \left(  t\right)  \nonumber \\
&  \leq3n\gamma \left \Vert P\right \Vert \mathrm{e}^{2\left(  \left \Vert
A\right \Vert +\left \Vert B\right \Vert ^{2}\left \Vert P\right \Vert \right)
\tau}\mathrm{e}^{2\left \Vert A\right \Vert \tau}%
%TCIMACRO{\dsum \limits_{i=1}^{3}}%
%BeginExpansion
{\displaystyle \sum \limits_{i=1}^{3}}
%EndExpansion
\left \Vert \varphi_{i}\left(  t\right)  \right \Vert ^{2},\label{eq30}%
\end{align}
and, by using the Jensen inequality again, we get, for $i=1,2,3,$%
\begin{align}
\left \Vert \varphi_{i}\left(  t\right)  \right \Vert ^{2} &  \leq \tau
\int_{t-\tau}^{t-\tau_{i}}u^{\mathrm{T}}\left(  s\right)  B_{i}^{\mathrm{T}%
}\mathrm{e}^{A^{\mathrm{T}}\left(  t-s-\tau_{i}\right)  }\mathrm{e}^{A\left(
t-s-\tau_{i}\right)  }B_{i}u\left(  s\right)  \mathrm{d}s\nonumber \\
&  \leq \left \Vert B_{i}\right \Vert ^{2}\mathrm{e}^{2\left \Vert A\right \Vert
\tau}\left(  \tau-\tau_{i}\right)  \int_{t-\tau}^{t-\tau_{i}}\left \Vert
u\left(  s\right)  \right \Vert ^{2}\mathrm{d}s\nonumber \\
&  \leq%
%TCIMACRO{\dsum \limits_{i=1}^{3}}%
%BeginExpansion
{\displaystyle \sum \limits_{i=1}^{3}}
%EndExpansion
\left \Vert B_{i}\right \Vert ^{2}\mathrm{e}^{2\left \Vert A\right \Vert \tau}%
\tau \int_{t-\tau}^{t}\left \Vert u\left(  s\right)  \right \Vert ^{2}%
\mathrm{d}s.\label{eq31}%
\end{align}
Moreover, it follows from (\ref{eq26}) that%
\begin{align}
\int_{t-\tau}^{t}\left \Vert u\left(  s\right)  \right \Vert ^{2}\mathrm{d}s &
\leq2\int_{t-\tau}^{t}\left \Vert \phi \left(  s\right)  \right \Vert
^{2}\mathrm{d}s+2\int_{t-\tau}^{t}y^{\mathrm{T}}\left(  s-\tau \right)
\mathrm{e}^{\left(  A+BF\right)  ^{\mathrm{T}}\tau}F^{\mathrm{T}}%
F\mathrm{e}^{\left(  A+BF\right)  \tau}y\left(  s-\tau \right)  \mathrm{d}%
s\nonumber \\
&  =2\int_{t-\tau}^{t}\left \Vert \phi \left(  s\right)  \right \Vert
^{2}\mathrm{d}s+2\int_{t-2\tau}^{t-\tau}y^{\mathrm{T}}\left(  s\right)
P\mathrm{e}^{-\left(  A+\gamma I\right)  \tau}BB^{\mathrm{T}}\mathrm{e}%
^{-\left(  A+\gamma I\right)  ^{\mathrm{T}}\tau}Py\left(  s\right)
\mathrm{d}s\nonumber \\
&  \leq2\int_{t-\tau}^{t}\left \Vert \phi \left(  s\right)  \right \Vert
^{2}\mathrm{d}s+2\left \Vert P\right \Vert \left \Vert \mathrm{e}^{-A\tau
}B\right \Vert ^{2}\int_{t-2\tau}^{t-\tau}y^{\mathrm{T}}\left(  s\right)
Py\left(  s\right)  \mathrm{d}s\nonumber \\
&  \leq2\int_{t-\tau}^{t}\left \Vert \phi \left(  s\right)  \right \Vert
^{2}\mathrm{d}s+2\left \Vert P\right \Vert \left \Vert \mathrm{e}^{-A\tau
}B\right \Vert ^{2}\int_{t-2\tau}^{t}y^{\mathrm{T}}\left(  s\right)  Py\left(
s\right)  \mathrm{d}s,\label{eq32}%
\end{align}
where we have used the inequality (\ref{eq29}). Inserting (\ref{eq32}) into
(\ref{eq31}) and then (\ref{eq30}) gives%
\begin{align}
\left \Vert \phi \left(  t\right)  \right \Vert ^{2} &  \leq9n\gamma
\tau \left \Vert P\right \Vert \mathrm{e}^{2\left(  \left \Vert A\right \Vert
+\left \Vert B\right \Vert ^{2}\left \Vert P\right \Vert \right)  \tau}%
\mathrm{e}^{4\left \Vert A\right \Vert \tau}%
%TCIMACRO{\dsum \limits_{i=1}^{3}}%
%BeginExpansion
{\displaystyle \sum \limits_{i=1}^{3}}
%EndExpansion
\left \Vert B_{i}\right \Vert ^{2}\int_{t-\tau}^{t}u^{\mathrm{T}}\left(
s\right)  u\left(  s\right)  \mathrm{d}s\nonumber \\
&  \leq c_{2}\gamma \left \Vert P\right \Vert \int_{t-\tau}^{t}\left \Vert
\phi \left(  s\right)  \right \Vert ^{2}\mathrm{d}s+c_{3}\gamma \left \Vert
P\right \Vert ^{2}\int_{t-2\tau}^{t}y^{\mathrm{T}}\left(  s\right)  Py\left(
s\right)  \mathrm{d}s,\label{34}%
\end{align}
where%
\begin{equation}
\left \{
\begin{array}
[c]{l}%
c_{2}=c_{2}\left(  \gamma \right)  =18n\tau \mathrm{e}^{2\left(  \left \Vert
A\right \Vert +\left \Vert B\right \Vert ^{2}\left \Vert P\right \Vert \right)
\tau}\mathrm{e}^{4\left \Vert A\right \Vert \tau}%
%TCIMACRO{\dsum \limits_{i=1}^{3}}%
%BeginExpansion
{\displaystyle \sum \limits_{i=1}^{3}}
%EndExpansion
\left \Vert B_{i}\right \Vert ^{2},\\
c_{3}=c_{3}\left(  \gamma \right)  =18n\tau \mathrm{e}^{2\left(  \left \Vert
A\right \Vert +\left \Vert B\right \Vert ^{2}\left \Vert P\right \Vert \right)
\tau}\mathrm{e}^{4\left \Vert A\right \Vert \tau}%
%TCIMACRO{\dsum \limits_{i=1}^{3}}%
%BeginExpansion
{\displaystyle \sum \limits_{i=1}^{3}}
%EndExpansion
\left \Vert B_{i}\right \Vert ^{2}\left \Vert \mathrm{e}^{-A\tau}B\right \Vert
^{2}.
\end{array}
\right.  \label{eq35}%
\end{equation}
Let $\gamma_{2}^{\ast}\in \left(  0,\gamma_{1}^{\ast}\right)  $ be such that%
\begin{equation}
c_{2}\gamma \left \Vert P\right \Vert \tau+\sqrt{\left \Vert P\right \Vert }%
\leq1,\; \forall \gamma \in \left(  0,\gamma_{2}^{\ast}\right)  .\label{eq36}%
\end{equation}
Then the time-derivative of the Lyapunov functional%
\begin{equation}
V_{3}\left(  \phi_{t}\right)  =c_{2}\gamma \left \Vert P\right \Vert \int_{-\tau
}^{0}\int_{t-s}^{t}\left \Vert \phi \left(  \sigma \right)  \right \Vert
^{2}\mathrm{d}\sigma \mathrm{d}s,\label{eq37}%
\end{equation}
satisfies%
\begin{align}
\dot{V}_{3}\left(  \phi_{t}\right)   &  =c_{2}\gamma \left \Vert P\right \Vert
\tau \left \Vert \phi \left(  t\right)  \right \Vert ^{2}-c_{2}\gamma \left \Vert
P\right \Vert \int_{t-\tau}^{t}\left \Vert \phi \left(  \sigma \right)
\right \Vert ^{2}\mathrm{d}\sigma \nonumber \\
&  \leq c_{2}\gamma \left \Vert P\right \Vert \tau \left \Vert \phi \left(
t\right)  \right \Vert ^{2}+c_{3}\gamma \left \Vert P\right \Vert ^{2}%
\int_{t-2\tau}^{t}y^{\mathrm{T}}\left(  s\right)  Py\left(  s\right)
\mathrm{d}s-\left \Vert \phi \left(  t\right)  \right \Vert ^{2}\nonumber \\
&  \leq c_{3}\gamma \left \Vert P\right \Vert ^{2}\int_{t-2\tau}^{t}%
y^{\mathrm{T}}\left(  s\right)  Py\left(  s\right)  \mathrm{d}s-\sqrt
{\left \Vert P\right \Vert }\left \Vert \phi \left(  t\right)  \right \Vert
^{2}.\label{eq38}%
\end{align}

Finally, we choose%
\begin{align}
V\left(  y_{t},\rho_{t},\phi_{t}\right)   &  =\frac{1}{8}\left \Vert
P\right \Vert V_{1}\left(  y\left(  t\right)  \right)  +\frac{1}{4}%
\sqrt{\left \Vert P\right \Vert }V_{2}\left(  \rho_{t}\right)  +V_{3}\left(
\phi_{t}\right) \nonumber \\
&  \quad+c_{3}\gamma \left \Vert P\right \Vert ^{2}\int_{-2\tau}^{0}\int
_{t-s}^{t}y^{\mathrm{T}}\left(  \sigma \right)  Py\left(  \sigma \right)
\mathrm{d}\sigma \mathrm{d}s, \label{eq39}%
\end{align}
and it follows from (\ref{eq47}), (\ref{eq44}) and (\ref{eq38}) that%
\begin{align}
\dot{V}\left(  y_{t},\rho_{t},\phi_{t}\right)   &  \leq \frac{1}{8}\left \Vert
P\right \Vert \left(  -\gamma y^{\mathrm{T}}\left(  t\right)  Py\left(
t\right)  +\left \Vert \rho \left(  t\right)  \right \Vert ^{2}\right)
\nonumber \\
&  \quad+\frac{1}{4}\left(  \sqrt{\left \Vert P\right \Vert }\left(  2\left \Vert
\phi \left(  t\right)  \right \Vert ^{2}-\sqrt{\left \Vert P\right \Vert
}\left \Vert \rho \left(  t\right)  \right \Vert ^{2}\right)  \right) \nonumber \\
&  \quad+c_{3}\gamma \left \Vert P\right \Vert ^{2}\int_{t-2\tau}^{t}%
y^{\mathrm{T}}\left(  s\right)  Py\left(  s\right)  \mathrm{d}s-\sqrt
{\left \Vert P\right \Vert }\left \Vert \phi \left(  t\right)  \right \Vert
^{2}\nonumber \\
&  \quad+2\tau c_{3}\gamma \left \Vert P\right \Vert ^{2}y^{\mathrm{T}}\left(
t\right)  Py\left(  t\right)  -c_{3}\gamma \left \Vert P\right \Vert ^{2}%
\int_{t-2\tau}^{t}y^{\mathrm{T}}\left(  s\right)  Py\left(  s\right)
\mathrm{d}s\nonumber \\
&  =-\frac{1}{8}\left \Vert P\right \Vert \left \Vert \rho \left(  t\right)
\right \Vert ^{2}-\frac{1}{2}\sqrt{\left \Vert P\right \Vert }\left \Vert
\phi \left(  t\right)  \right \Vert ^{2}\nonumber \\
&  \quad-\left(  \frac{1}{8}-2\tau c_{3}\left \Vert P\right \Vert \right)
\left \Vert P\right \Vert \gamma y^{\mathrm{T}}\left(  t\right)  Py\left(
t\right)  . \label{eq40}%
\end{align}
Hence there exists a $\gamma^{\ast}\in \left(  0,\gamma_{2}^{\ast}\right)  $
such that%
\begin{align}
\dot{V}\left(  y_{t},\rho_{t},\phi_{t}\right)   &  \leq-\frac{1}{8}\left \Vert
P\right \Vert \left \Vert \rho \left(  t\right)  \right \Vert ^{2}-\frac{1}%
{2}\sqrt{\left \Vert P\right \Vert }\left \Vert \phi \left(  t\right)  \right \Vert
^{2}\nonumber \\
&  \quad-\frac{1}{16}\left \Vert P\right \Vert \gamma y^{\mathrm{T}}\left(
t\right)  Py\left(  t\right)  ,\; \forall \gamma \in \left(  0,\gamma^{\ast
}\right)  . \label{eq41}%
\end{align}
Then by a similar argument in \cite{zld10siamjco}, we can show that
$\lim_{t\rightarrow \infty}\left \Vert \rho \left(  t\right)  \right \Vert
=\lim_{t\rightarrow \infty}\left \Vert \phi \left(  t\right)  \right \Vert
=\lim_{t\rightarrow \infty}\left \Vert y\left(  t\right)  \right \Vert =0.$
Consequently, by (\ref{eq6}) we know $\lim_{t\rightarrow \infty}\left \Vert
x\left(  t\right)  \right \Vert =0$, which implies asymptotic stability. The
proof is finished.
\end{proof}

\begin{remark}
We point out that all the plant parameters and the delays in system
(\ref{sys00}) can be time-varying as in \cite{zhou14auto} and similar results
can be obtained. We here however only consider time-invariant parameters for
the sake of brevity.
\end{remark}

\section{\label{sec4}Applications to the Consensus of Multi-Agent Systems}

\subsection{Problem Formulation and Assumptions}

We consider a continuous-time high-order multi-agent system described by%
\begin{equation}
\dot{x}_{i}\left(  t\right)  =Ax_{i}\left(  t\right)  +%
%TCIMACRO{\dsum \limits_{j=1}^{3}}%
%BeginExpansion
{\displaystyle \sum \limits_{j=1}^{3}}
%EndExpansion
B_{j}u_{i}\left(  t-\tau_{j}\right)  ,\;i\in \mathbf{I}\left[  1,N\right]  ,
\label{agents}%
\end{equation}
where $x_{i}\in \mathbf{R}^{n}$ and $u_{i}\in \mathbf{R}^{m}$ are the state and
the control of Agent $i,$ respectively, $N\geq1$ is a given integer denoting
the number of agents, $\tau_{j},j=1,2,3$ are the input delays satisfying
(\ref{delay1}), and $A,B_{i}$ are given matrices. Let the communication
topology among these agents be characterized by a weighted directed graph
$\mathcal{G}\left(  \mathcal{N},\mathcal{E},\mathcal{A}\right)  ,$ where
$\mathcal{N}$ is the node set, $\mathcal{E}$ is the edge set, and
$\mathcal{A}=\left[  \alpha_{ij}\right]  \in \mathbf{R}^{N\times N}$ is the
weighted adjacency matrix. Denote the corresponding Laplacian by $L=\left[
l_{ij}\right]  \in \mathbf{R}^{N\times N}$ (see \cite{rc10book} for a detailed
introduction on graph theory).

We assume that Agent $i$ collects the state information of its neighboring
agents by the rule \cite{zl14auto}%
\begin{align}
z_{i}\left(  t\right)   &  =\sum \limits_{j\in N_{i}}\alpha_{ij}\left(
x_{i}\left(  t\right)  -x_{j}\left(  t\right)  \right) \nonumber \\
&  =\sum \limits_{j=1}^{N}l_{ij}x_{j}\left(  t\right)  ,\;i\in \mathbf{I}\left[
1,N\right]  , \label{eq21a}%
\end{align}
where $N_{i}=\{j:\alpha_{ij}\neq0\}.$ In this section, we are interested in
the design of the state feedback protocol $u_{i}\left(  t\right)
=u_{i}\left(  z_{i}\right)  ,i\in \mathbf{I}\left[  1,N\right]  $ that achieves
consensus of the multi-agent system (\ref{agents}). To this end, we first
introduce the concept of consensus.

\begin{definition}
The linear multi-agent system (\ref{agents}) achieves consensus if
$\lim_{t\rightarrow \infty}\left \Vert x_{i}\left(  t\right)  -x_{j}\left(
t\right)  \right \Vert =0,\; \forall i,j\in \mathbf{I}\left[  1,N\right]  .$
\end{definition}

\begin{remark}
\label{rm1}By definition, if the high-order multi-agent system (\ref{agents})
achieves consensus, then there exists a signal $s\left(  t\right)
\in \mathbf{R}^{n}$, which may be unbounded, such that $\lim_{t\rightarrow
\infty}\left \Vert x_{i}\left(  t\right)  -s\left(  t\right)  \right \Vert =0,\;
\forall i\in \mathbf{I}\left[  1,N\right]  .$ The signal $s\left(  t\right)  $
is referred to as the reference trajectory.
\end{remark}

According to the discussions in \cite{ldch10tcasi} and \cite{zl14auto}, to
ensure that the consensus value reached by the agents will not tend to
infinity exponentially, namely, the reference trajectory $s\left(  t\right)  $
defined in Remark \ref{rm1} is not exponentially diverging, the matrix $A$
should not contain eigenvalues in the open right-half plane \cite{ldch10tcasi}%
. Hence, without loss of generality, we assume that Assumption \ref{ass2} is
satisfied for (\ref{agents}) \cite{zl14auto}.

The following assumption is also necessary for the solvability of the
consensus problem (see, for example, \cite{zl14auto}).

\begin{assumption}
\label{ass3}The communication topology $\mathcal{G}\left(  \mathcal{N}%
,\mathcal{E},\mathcal{A}\right)  $ contains a directed spanning tree.
\end{assumption}

By this assumption we know that $\lambda \left(  L\right)  =\{0,\lambda
_{2},\ldots,\lambda_{N}\}$ with $\operatorname{Re}\left \{  \lambda
_{i}\right \}  >0,i\in \mathbf{I}\left[  2,N\right]  $ (see, for example,
\cite{rc10book}).

\subsection{Consensus by Truncated Pseudo-Predictor Feedback Protocols}

In this section we present our solutions to the consensus problem by using the
TPPF approach developed in Section \ref{sec3} for stabilization of the linear
system (\ref{sys1}).

\begin{theorem}
\label{th3}Let Assumptions \ref{ass2} and \ref{ass3} be satisfied,
$F=-B^{\mathrm{T}}P$ where $P>0$ is the unique positive definite solution to
the parametric ARE (\ref{are}) and $\mu$ be any positive number such that
\begin{equation}
\mu \geq \max_{i\in \mathbf{I}\left[  2,N\right]  }\left \{  \frac{1}%
{\operatorname{Re}\left \{  \lambda_{i}\right \}  }\right \}  . \label{eqmu}%
\end{equation}
Then, for any delay $\tau$ that is exactly known and can be arbitrarily large
yet bounded, there exists a number $\gamma^{\ast}=\gamma^{\ast}(\mu
,\tau,\left \{  \lambda_{i}\right \}  _{i=2}^{N})$ such that the consensus of
the multi-agent system (\ref{agents}) is achieved by the protocol%
\begin{equation}
u_{i}\left(  t\right)  =\mu F\mathrm{e}^{\left(  A+BF\right)  \tau}%
\mathrm{e}^{-A\tau}z_{i}\left(  t\right)  ,\;i\in \mathbf{I}\left[  1,N\right]
,\; \gamma \in(0,\gamma^{\ast}). \label{protocol}%
\end{equation}
Moreover, the reference signal $s\left(  t\right)  $ defined in Remark
\ref{rm1} satisfies%
\begin{equation}
\left \Vert s\left(  t\right)  \right \Vert \leq k\left(  1+t^{N^{\ast}%
-1}\right)  ,\; \forall t\geq0, \label{eq34}%
\end{equation}
where $N^{\ast}$ is the maximal algebraic multiplicity of the eigenvalues of
$A$ and $k$ is a positive function of the initial conditions.
\end{theorem}

\begin{proof}
Under Assumption \ref{ass3}, there exists a nonsingular matrix $U\in
\mathbf{C}^{N\times N},$ whose first column is $\mathbf{1}_{N}\triangleq
\left[  1,1,\ldots,1\right]  ^{\intercal}\in \mathbf{R}^{N},$ such that
\cite{rc10book}%
\begin{equation}
U^{-1}LU=\left[
\begin{array}
[c]{ccccc}%
0 &  &  &  & \\
& \lambda_{2} & \delta_{2} &  & \\
&  & \ddots & \ddots & \\
&  &  & \lambda_{N-1} & \delta_{N-1}\\
&  &  &  & \lambda_{N}%
\end{array}
\right]  \triangleq J, \label{cssjl}%
\end{equation}
where $\lambda_{i}$ are such that $\operatorname{Re}\{ \lambda_{i}%
\}>0,i\in \mathbf{I}\left[  2,N\right]  ,$ and $\delta_{i},i\in \mathbf{I}%
\left[  2,N-1\right]  ,$ equals either 1 or 0. For future use, we define
$\delta_{N}=0.$

Let $\mathscr{F}=\mu F\mathrm{e}^{\left(  A+BF\right)  \tau}\mathrm{e}%
^{-A\tau}.$ The closed-loop system consisting of (\ref{agents}) and
(\ref{protocol}) is given by%
\begin{align}
\dot{x}_{i}\left(  t\right)   &  =Ax_{i}\left(  t\right)  +%
%TCIMACRO{\dsum \limits_{k=1}^{3}}%
%BeginExpansion
{\displaystyle \sum \limits_{k=1}^{3}}
%EndExpansion
B_{k}\mathscr{F}z_{i}\left(  t-\tau_{k}\right) \nonumber \\
&  =Ax_{i}\left(  t\right)  +%
%TCIMACRO{\dsum \limits_{k=1}^{3}}%
%BeginExpansion
{\displaystyle \sum \limits_{k=1}^{3}}
%EndExpansion
B_{k}\mathscr{F}\sum \limits_{j=1}^{N}l_{ij}x_{j}\left(  t-\tau_{k}\right)
\nonumber \\
&  =Ax_{i}\left(  t\right)  +%
%TCIMACRO{\dsum \limits_{k=1}^{3}}%
%BeginExpansion
{\displaystyle \sum \limits_{k=1}^{3}}
%EndExpansion
\sum \limits_{j=1}^{N}l_{ij}B_{k}\mathscr{F}x_{j}\left(  t-\tau_{k}\right)
,\;i\in \mathbf{I}\left[  1,N\right]  . \label{eq51}%
\end{align}
By letting $x=\left[  x_{1}^{\intercal},x_{2}^{\intercal},\cdots
,x_{N}^{\intercal}\right]  ^{\intercal},$ the above system can be expressed
as
\begin{align}
\dot{x}\left(  t\right)   &  =\left(  I_{N}\otimes A\right)  x\left(
t\right)  +%
%TCIMACRO{\dsum \limits_{k=1}^{3}}%
%BeginExpansion
{\displaystyle \sum \limits_{k=1}^{3}}
%EndExpansion
\left(  L\otimes B_{k}\mathscr{F}\right)  x\left(  t-\tau_{k}\right)
\nonumber \\
&  =\left(  U\otimes I_{n}\right)  \left(  I_{N}\otimes A\right)  \left(
U^{-1}\otimes I_{n}\right)  x\left(  t\right) \nonumber \\
&  \quad+\left(  U\otimes I_{n}\right)
%TCIMACRO{\dsum \limits_{k=1}^{3}}%
%BeginExpansion
{\displaystyle \sum \limits_{k=1}^{3}}
%EndExpansion
\left(  J\otimes B_{k}\mathscr{F}\right)  \left(  U^{-1}\otimes I_{n}\right)
x\left(  t-\tau_{k}\right)  , \label{eq56}%
\end{align}
namely,%
\begin{equation}
\dot{\chi}\left(  t\right)  =\left(  I_{N}\otimes A\right)  \chi \left(
t\right)  +%
%TCIMACRO{\dsum \limits_{k=1}^{3}}%
%BeginExpansion
{\displaystyle \sum \limits_{k=1}^{3}}
%EndExpansion
\left(  J\otimes B_{k}\mathscr{F}\right)  \chi \left(  t-\tau_{k}\right)  ,
\label{eq57}%
\end{equation}
where $\chi=\left(  U^{-1}\otimes I_{n}\right)  x=\left[  \chi_{1}^{\intercal
},\chi_{2}^{\intercal},\cdots,\chi_{N}^{\intercal}\right]  ^{\intercal}.$ In
view of (\ref{34}), system (\ref{eq57}) is equivalent to%
\begin{equation}
\left \{
\begin{array}
[c]{ll}%
\dot{\chi}_{i}\left(  t\right)  =A\chi_{i}\left(  t\right)  , & i=1,\\
\dot{\chi}_{i}\left(  t\right)  =A\chi_{i}\left(  t\right)  +\lambda_{i}%
%TCIMACRO{\dsum \limits_{k=1}^{3}}%
%BeginExpansion
{\displaystyle \sum \limits_{k=1}^{3}}
%EndExpansion
B_{k}\mathscr{F}\chi_{i}\left(  t-\tau_{k}\right)  +\delta_{i}\chi
_{i+1}\left(  t-\tau_{k}\right)  , & i\in \mathbf{I}\left[  2,N\right]  ,
\end{array}
\right.  \label{eq58}%
\end{equation}
where $\chi_{N+1}\left(  t\right)  \equiv0.$ Since $x=\left(  U\otimes
I_{n}\right)  \chi,$ it follows from the special structure of $U$ that the
consensus is achieved if
\begin{equation}
\lim_{t\rightarrow \infty}\left \Vert \chi_{i}\left(  t\right)  \right \Vert
=0,\;i\in \mathbf{I}\left[  2,N\right]  , \label{eq59}%
\end{equation}
and in this case
\begin{equation}
x_{i}\left(  t\right)  \overset{t\rightarrow \infty}{\longrightarrow}\chi
_{1}\left(  t\right)  ,\;i\in \mathbf{I}\left[  1,N\right]  , \label{eq60}%
\end{equation}
which implies (\ref{eq34}) in view of the first equation of (\ref{eq58}). It
is easy to see that (\ref{eq59}) is true if and only if%
\begin{equation}
\dot{\chi}_{i}\left(  t\right)  =A\chi_{i}\left(  t\right)  +\lambda_{i}%
%TCIMACRO{\dsum \limits_{k=1}^{3}}%
%BeginExpansion
{\displaystyle \sum \limits_{k=1}^{3}}
%EndExpansion
B_{k}\mathscr{F}\chi_{i}\left(  t-\tau_{k}\right)  ,\;i\in \mathbf{I}\left[
2,N\right]  , \label{eq61}%
\end{equation}
are all asymptotically stable. In the following we will show that (\ref{eq61})
are all indeed asymptotically stable if (\ref{eqmu}) is satisfied.

We consider the following linear time-delay system%
\begin{equation}
\left \{
\begin{array}
[c]{l}%
\dot{\varkappa}\left(  t\right)  =A\varkappa \left(  t\right)  +%
%TCIMACRO{\dsum \limits_{k=1}^{3}}%
%BeginExpansion
{\displaystyle \sum \limits_{k=1}^{3}}
%EndExpansion
B_{k}u\left(  t-\tau_{k}\right)  ,\\
u\left(  t\right)  =\mathcal{F}\mathrm{e}^{\left(  A+BF\right)  \tau
}\mathrm{e}^{-A\tau}\varkappa \left(  t\right)  ,
\end{array}
\right.  \label{sys3}%
\end{equation}
where
\begin{equation}
\mathcal{F}=\lambda \mu F,\; \mu \geq \frac{1}{\operatorname{Re}\left \{
\lambda \right \}  },\; \operatorname{Re}\left \{  \lambda \right \}  >0.
\label{sys4}%
\end{equation}
By comparing (\ref{eq61}) with (\ref{sys3})-(\ref{sys4}), it suffices to show
that there exists a $\gamma^{\#}\left(  \lambda \right)  $ such that
(\ref{sys3}) is asymptotically stable for all $\gamma \in(0,\gamma^{\#}\left(
\lambda \right)  ).$ However, system (\ref{sys3}) is very similar to the
closed-loop system consisting of (\ref{sys1}) and the TPPF (\ref{tppf}). Hence
we need only to prove the stability of%
\begin{equation}
\dot{y}\left(  t\right)  =\left(  A+B\mathcal{F}\right)  y\left(  t\right)
+B\rho \left(  t\right)  , \label{eq69}%
\end{equation}
where%
\begin{equation}
\left \{
\begin{array}
[c]{l}%
\rho \left(  t\right)  =\mathcal{F}\left(  \mathrm{e}^{\left(  A+BF\right)
\tau}y\left(  t-\tau \right)  -y\left(  t\right)  \right)  +\phi \left(
t\right)  ,\\
\phi \left(  t\right)  =\mathcal{F}\mathrm{e}^{\left(  A+BF\right)  \tau
}\mathrm{e}^{-A\tau}\varphi \left(  t\right)  ,\\
\varphi \left(  t\right)  =%
%TCIMACRO{\dsum \limits_{i=1}^{3}}%
%BeginExpansion
{\displaystyle \sum \limits_{i=1}^{3}}
%EndExpansion
\int_{t-\tau}^{t-\tau_{i}}\mathrm{e}^{A\left(  t-s-\tau_{i}\right)  }%
B_{i}u\left(  s\right)  \mathrm{d}s.
\end{array}
\right.  \label{eq70}%
\end{equation}
The proof is almost the same as the proof of Theorem \ref{th1} where the
constants $c_{i},i=1,2,3,$ are redefined accordingly. Particularly, the
inequality in (\ref{eq47}) can be obtained since
\begin{align}
\left(  A+B\mathcal{F}\right)  ^{\mathrm{H}}P+P\left(  A+B\mathcal{F}\right)
&  =A^{\mathrm{T}}P+PA-2\operatorname{Re}\left \{  \lambda \right \}  \mu
PBB^{\mathrm{T}}P\nonumber \\
&  \leq A^{\mathrm{T}}P+PA-2PBB^{\mathrm{T}}P\nonumber \\
&  =-\gamma P-PBB^{\mathrm{T}}P, \label{eq71}%
\end{align}
where we have used (\ref{sys4}).

Based on the above discussion, the proof is finished by choosing $\gamma
^{\ast}=\min_{i\in \mathbf{I}\left[  2,N\right]  }\left \{  \gamma^{\#}\left(
\lambda_{i}\right)  \right \}  .$
\end{proof}

The TPPF protocols can also be generalized to the observer based output
feedback case. The details are omitted for brevity.

\section{\label{sec5}Two Illustrative Examples}

\subsection{Stabilization by Truncated Pseudo-Predictor Feedback}

Consider a linear system in the following form%
\begin{equation}
\dot{x}\left(  t\right)  =Ax\left(  t\right)  +B_{1}u\left(  t\right)
+B_{2}u\left(  t-\tau \right)  , \label{sysex}%
\end{equation}
where $\tau=\frac{\pi}{2}$ and%
\begin{equation}
A=\left[
\begin{array}
[c]{cccc}%
0 & 1 & 0 & 0\\
-1 & 0 & 1 & 0\\
0 & 0 & 0 & 1\\
0 & 0 & -1 & 0
\end{array}
\right]  ,\;B_{1}=\left[
\begin{array}
[c]{c}%
0\\
-1\\
0\\
0
\end{array}
\right]  ,\;B_{2}=\left[
\begin{array}
[c]{c}%
0\\
0\\
0\\
-1
\end{array}
\right]  . \label{ABB}%
\end{equation}
It follows that $\lambda \left(  A\right)  =\{ \pm \mathrm{i},\pm \mathrm{i}\}$
and $\left(  A,B\right)  $ is controllable, where $B$ is found to be%
\begin{equation}
B=\left[
\begin{array}
[c]{cccc}%
-\frac{1}{2} & \frac{1}{4}-\pi & -1 & 0
\end{array}
\right]  ^{\mathrm{T}}. \label{BB}%
\end{equation}
By Lemma \ref{lm0} and Theorem \ref{th1}, both the TPF and the TPPF
controllers take the form%
\begin{equation}
u\left(  t\right)  =\mathscr{F}\left(  \gamma \right)  x\left(  t\right)  ,
\label{2u}%
\end{equation}
where $\mathscr{F}\left(  \gamma \right)  =F_{\mathrm{TPF}}$ for the TPF and
$\mathscr{F}=F_{\mathrm{TPPF}}$ for the TPPF with $F_{\mathrm{TPF}}$ and
$F_{\mathrm{TPPF}}$ given by
\begin{align}
F_{\mathrm{TPF}}  &  =\left[
\begin{array}
[c]{cccc}%
4{\gamma}^{3} & -{\gamma}^{4}+4{\gamma}^{2} & 4\gamma+{\gamma}^{4}-2{\gamma
}^{3}-4{\gamma}^{2}-\frac{1}{4}\pi{\gamma}^{4}+\pi{\gamma}^{2} & -\frac{1}%
{2}{\gamma}^{4}+4{\gamma}^{3}-4{\gamma}^{2}-{\gamma}^{3}\pi
\end{array}
\right]  ,\label{ftpf}\\
F_{\mathrm{TPPF}}  &  =\frac{1}{4}\gamma \mathrm{e}{^{-\frac{1}{2}\gamma \pi}%
}\left[
\begin{array}
[c]{c}%
{\gamma}^{2}\left(  4\gamma \pi-3{\gamma}^{3}\pi+2{\gamma}^{4}+8{\gamma}%
^{2}-16\right) \\
2{\gamma}\left(  -4\gamma \pi-5{\gamma}^{3}\pi+3{\gamma}^{4}+18{\gamma}%
^{2}+8\right) \\
16+6{\gamma}^{4}-36{\gamma}^{3}+40{\gamma}^{2}-16\gamma-6{\gamma}^{5}%
-4\gamma \pi+10\pi{\gamma}^{4}-11{\gamma}^{3}\pi+8\pi{\gamma}^{2}\\
{\gamma}\left(  8{\gamma}^{3}+2{\gamma}^{2}-16\gamma+16+2{\gamma}^{5}%
+4\pi{\gamma}^{2}-4\gamma \pi-3\pi{\gamma}^{4}\right)
\end{array}
\right]  ^{\mathrm{T}}. \label{ftppf}%
\end{align}

The closed-loop system consisting of (\ref{sysex}) and (\ref{2u}) is given by%
\begin{equation}
\dot{x}\left(  t\right)  =\left(  A+B_{1}\mathscr{F}\left(  \gamma \right)
\right)  x\left(  t\right)  +B_{2}\mathscr{F}\left(  \gamma \right)  x\left(
t-\tau \right)  , \label{ppfclosed1}%
\end{equation}
which is a linear time-delay system whose stability is completely determined
by the right-most roots of its characteristic equation $c\left(
s,\gamma \right)  =\det \left(  sI_{4}-\left(  A+B_{1}\mathscr{F}\left(
\gamma \right)  \right)  -B_{2}\mathscr{F}\left(  \gamma \right)  \mathrm{e}%
^{-\tau s}\right)  $ \cite{hale71book,zld12auto}. Let \cite{zhou14auto}
\begin{equation}
\lambda_{\max}\left(  \gamma \right)  =\max \left \{  \operatorname{Re}\left \{
s\right \}  :c\left(  s,\gamma \right)  =0\right \}  , \label{tveq121}%
\end{equation}
which can be efficiently computed by the software package DDE-BIFTOOL
\cite{els01soft} for a fixed $\gamma$. Then the closed-loop system
(\ref{ppfclosed1}) is asymptotically stable if and only if $\lambda_{\max
}\left(  \gamma \right)  <0\ $and its convergence rate is completely determined
by $\lambda_{\max}\left(  \gamma \right)  $ \cite{hale71book,zld12auto}. Let
$\gamma_{\sup}>0$ be the minimal number such that $\lambda_{\max}\left(
\gamma \right)  =0$. Then clearly there exists$\ $an optimal value
$\gamma_{\mathrm{opt}}\in \left(  0,\gamma_{\sup}\right)  $ such that
$\lambda_{\max}\left(  \gamma \right)  $ is minimized at $\gamma=\gamma
_{\mathrm{opt}}$ with the minimal value $\lambda_{\max \min}$. It is thus clear
that $\lambda_{\max \min}$ is the maximal convergence rate that the controller
(\ref{2u}) can achieve \cite{zhou14book,zld12auto,zhou14auto}.

The functions $\lambda_{\max}\left(  \gamma \right)  $ associated with both
$\mathscr{F}\left(  \gamma \right)  =F_{\mathrm{TPF}}$ and $\mathscr{F}\left(
\gamma \right)  =F_{\mathrm{TPPF}}$ are shown in Fig. \ref{ppffig1} from which
we clearly see that

\begin{figure}[ptb]
\begin{center}
\includegraphics[scale=0.65]{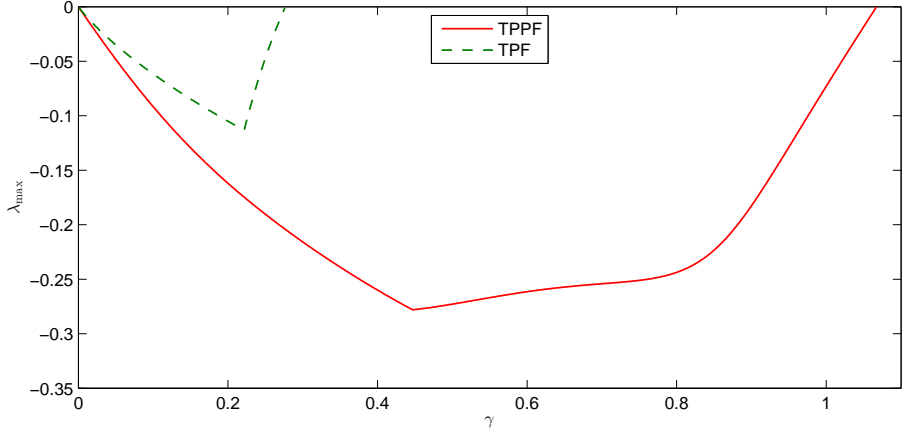}
\end{center}
\caption{The functions $\lambda_{\max}\left(  \gamma \right)  $ for the
closed-loop system (\ref{ppfclosed1}) under the TPF and the TPPF}%
\label{ppffig1}%
\end{figure}

\begin{itemize}
\item[A.] The closed-loop system under the TPF is asymptotically stable if and
only if $\gamma \in(0,\gamma_{\sup}^{\mathrm{TPF}})$ where $\gamma_{\sup
}^{\mathrm{TPF}}=0.276,$ and its convergence rate is maximized with
$\gamma=\gamma_{\mathrm{opt}}^{\mathrm{TPF}}=0.221$ and $\lambda_{\max \min
}^{\mathrm{TPF}}=\lambda_{\max}(\gamma_{\mathrm{opt}}^{\mathrm{TPF}})=-0.113,$
which is the maximal convergence rate that the TPF can achieve.

\item[B.] The closed-loop system under the TPPF is asymptotically stable if
and only if $\gamma \in(0,\gamma_{\sup}^{\mathrm{TPPF}})$ where $\gamma_{\sup
}^{\mathrm{TPPF}}=1.067$ and its convergence rate is maximized with
$\gamma=\gamma_{\mathrm{opt}}^{\mathrm{TPPF}}=0.447$ and $\lambda_{\max \min
}^{\mathrm{TPPF}}=\lambda_{\max}(\gamma_{\mathrm{opt}}^{\mathrm{TPPF}%
})=-0.278,$ which is the maximal convergence rate that the TPPF can achieve.
\end{itemize}

We thus conclude that the TPPF can not only allow a larger range of $\gamma$
ensuring the stability than the TPF (namely, $\gamma_{\sup}^{\mathrm{TPPF}%
}>\gamma_{\sup}^{\mathrm{TPF}}$) but also lead to a much better control
performance than the TPF (namely, $\lambda_{\max \min}^{\mathrm{TPPF}}%
<\lambda_{\max \min}^{\mathrm{TPF}}$). For the comparison purpose, both the
state responses of the closed-loop system by TPF with $\gamma=\gamma
_{\mathrm{opt}}^{\mathrm{TPF}}$ and the closed-loop system by the TPPF with
$\gamma=\gamma_{\mathrm{opt}}^{\mathrm{TPPF}}$ are recorded in Fig.
\ref{ppffig2} from which we can see that the later one indeed converges faster
than the former one.

\begin{figure}[ptb]
\begin{center}
\includegraphics[scale=0.65]{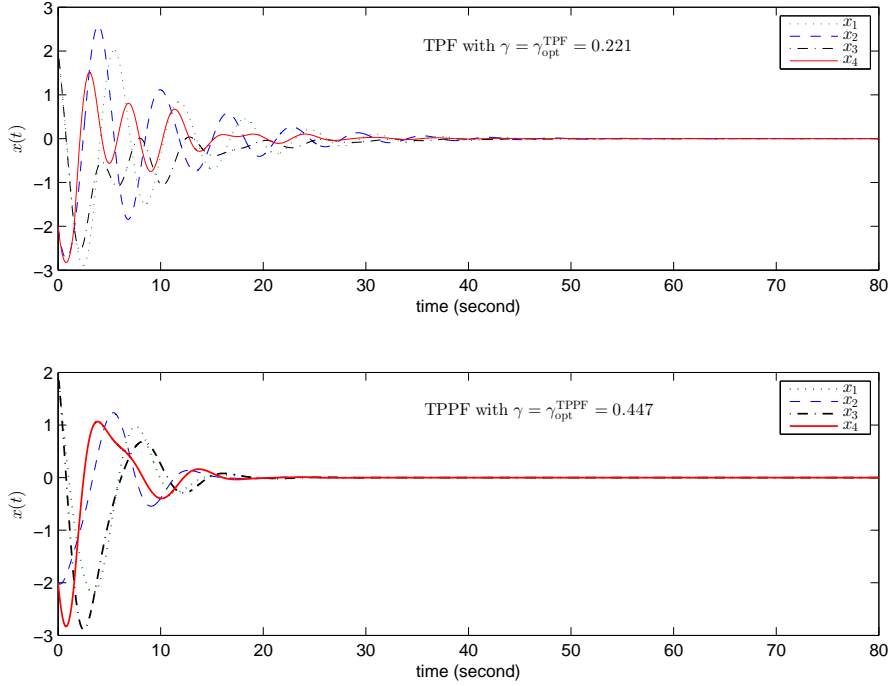}
\end{center}
\caption{The state responses of the closed-loop systems (\ref{ppfclosed1}) by
the TPF with $\gamma=\gamma_{\mathrm{opt}}^{\mathrm{TPF}}$ and by the TPPF
with $\gamma=\gamma_{\mathrm{opt}}^{\mathrm{TPPF}}$}%
\label{ppffig2}%
\end{figure}

\subsection{Consensus by Truncated Pseudo-Predictor Feedback}

We assume that there are six agents with each agent being characterized by the
linear system (\ref{sysex})--(\ref{ABB}). Let these agents be connected by a
communication network as shown in Fig. \ref{ppffig3}. This network can be
characterized by the Laplacian%
\begin{equation}
L=\left[
\begin{array}
[c]{cccccc}%
2 & -2 & 0 & 0 & 0 & 0\\
-1 & 1 & 0 & 0 & 0 & 0\\
0 & 0 & 2 & 0 & -2 & 0\\
0 & 0 & -1 & 1 & 0 & 0\\
-2 & 0 & 0 & 0 & 2 & 0\\
0 & 0 & 0 & 0 & -1 & 1
\end{array}
\right]  . \label{cssL1}%
\end{equation}
As $\lambda \left(  L\right)  =\{0,1,1,2,2,3\}$, Assumption \ref{ass3} is
fulfilled. We can thus choose $\mu=1.$ By Theorem \ref{th3}, the protocol can
then be designed as
\begin{equation}
u_{i}\left(  t\right)  =\mathscr{F}\left(  \gamma \right)  z_{i}\left(
t\right)  ,\;i\in \mathbf{I}\left[  1,6\right]  , \label{protocols}%
\end{equation}
where $\mathscr{F}\left(  \gamma \right)  $ is given by (\ref{ftppf}) and
$\gamma=0.1.$ For some specified initial conditions, the state vector $x_{1}$
of the first agent and the differences between the states $x_{i}$ of the other
five agents and $x_{1}$ are recorded in Fig. \ref{ppffig5}. From this figure,
we clearly see that the consensus is achieved.

\begin{figure}[ptb]
\begin{center}
\includegraphics[scale=0.38]{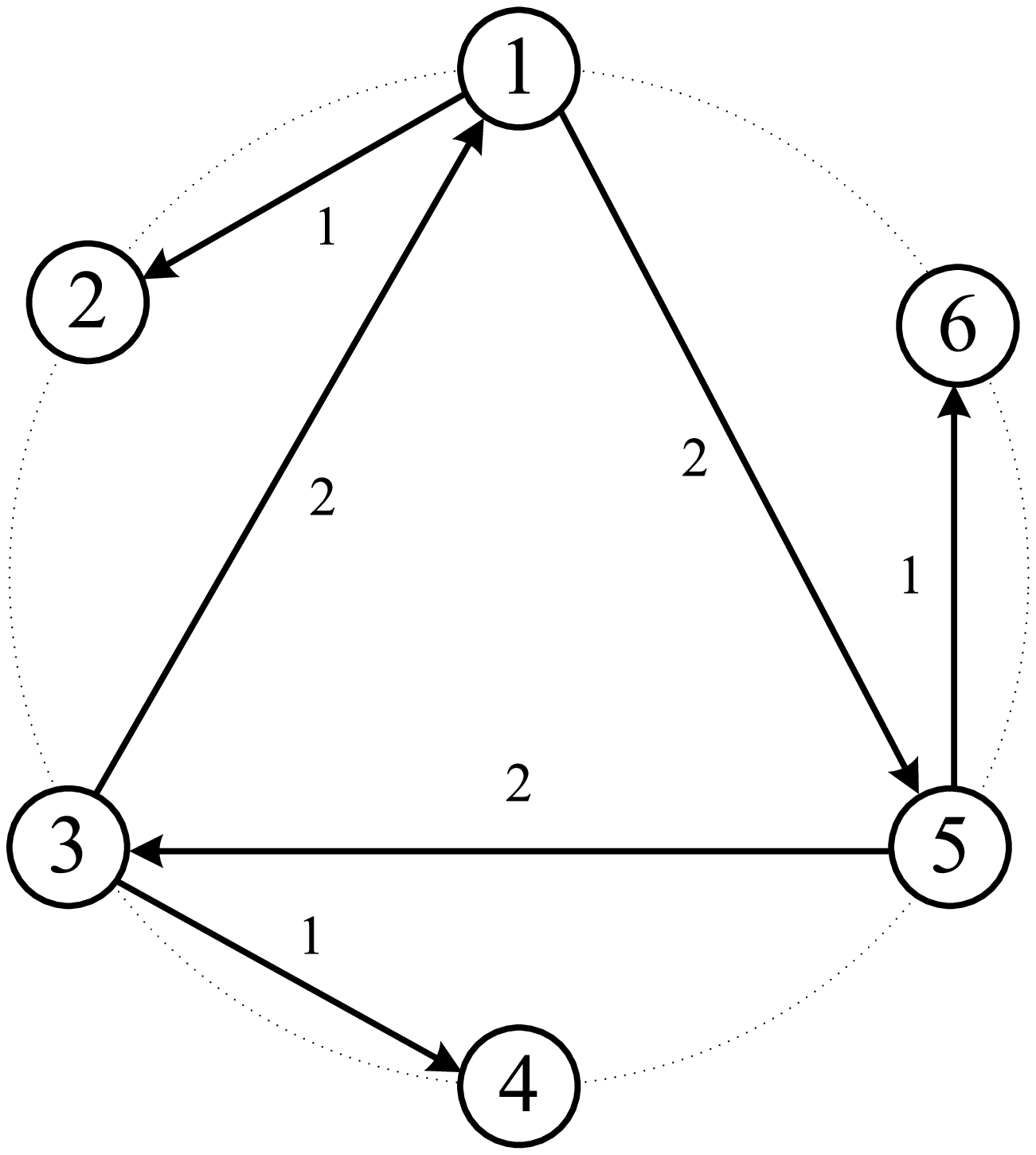}
\end{center}
\caption{The communication network}%
\label{ppffig3}%
\end{figure}

\section{\label{sec6}Conclusion}

This paper has studied the stabilization problem for linear systems with
multiple input delays. We first establish a pseudo-predictor feedback
controller, which is infinite dimensional, and then build a truncated version
of this pseudo-predictor feedback controller, which is finite dimensional and
is termed as truncated pseudo-predictor feedback (TPPF). It is shown that the
TPPF can stabilize the concerned systems for arbitrarily large yet bounded
delay if the open-loop system is not exponentially stable. Comparison of the
TPPF with the existing results shows that the former one can generally lead to
better control performance. The TPPF approach has also been utilized to solve
the consensus problem for linear multi-agent systems with multiple input
delays. \begin{figure}[ptb]
\begin{center}
\includegraphics[scale=0.65]{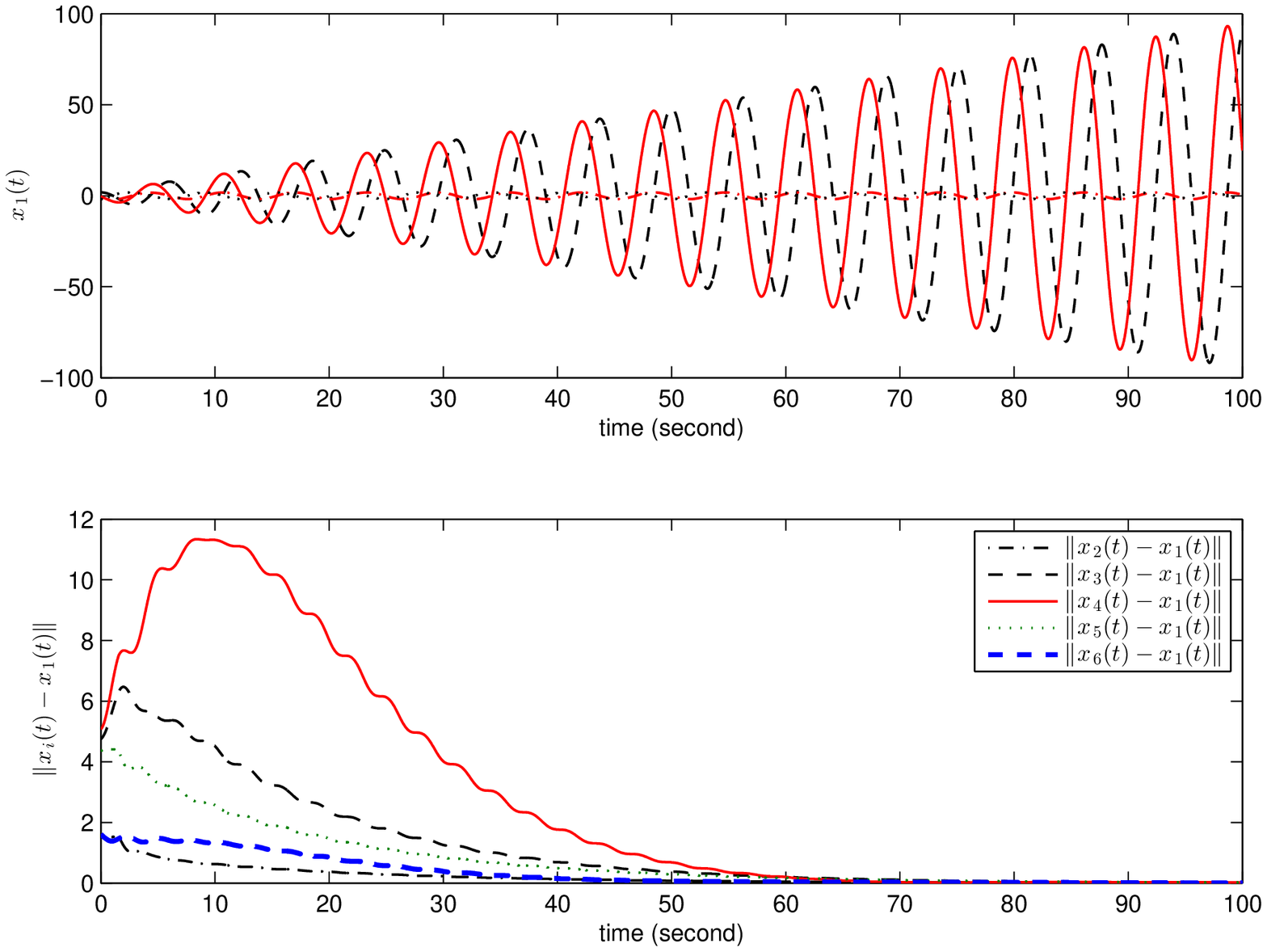}
\end{center}
\caption{The state vector $x_{1}(t)$ and $\left \Vert x_{i}(t)-x_{1}%
(t)\right \Vert ,i\in \mathbf{I}\left[  2,6\right]  $}%
\label{ppffig5}%
\end{figure}

\section*{Appendix}

\subsection*{A1. Proof of Lemma \ref{lm1}}

We compute from (\ref{y3}), (\ref{y2}), (\ref{y1}) and (\ref{yt}) that%
\begin{align}
y\left(  t\right)   &  =\mathrm{e}^{-A\tau}y_{1}\left(  t\right) \nonumber \\
&  =\mathrm{e}^{-A\tau}\left(  \mathrm{e}^{A\tau_{1}}y_{2}\left(  t\right)
+\int_{t}^{t+\tau_{1}}\mathrm{e}^{A\left(  t+\tau_{1}-s\right)  }\left(
\mathrm{e}^{A\left(  \tau_{2}+\tau_{3}\right)  }B_{1}u\left(  s-\tau
_{1}\right)  \right)  \mathrm{d}s\right) \nonumber \\
&  =\mathrm{e}^{-A\left(  \tau_{2}+\tau_{3}\right)  }y_{2}\left(  t\right)
+\int_{t}^{t+\tau_{1}}\mathrm{e}^{A\left(  t-s\right)  }B_{1}u\left(
s-\tau_{1}\right)  \mathrm{d}s\nonumber \\
&  =\mathrm{e}^{-A\left(  \tau_{2}+\tau_{3}\right)  }\left(  \mathrm{e}%
^{A\tau_{2}}y_{3}\left(  t\right)  +\int_{t}^{t+\tau_{2}}\mathrm{e}^{A\left(
t+\tau_{2}-s\right)  }\left(  \mathrm{e}^{A\tau_{3}}B_{2}u\left(  s-\tau
_{2}\right)  \right)  \mathrm{d}s\right)  +\int_{t}^{t+\tau_{1}}%
\mathrm{e}^{A\left(  t-s\right)  }B_{1}u\left(  s-\tau_{1}\right)
\mathrm{d}s\nonumber \\
&  =\mathrm{e}^{-A\tau_{3}}y_{3}\left(  t\right)  +\int_{t}^{t+\tau_{2}%
}\mathrm{e}^{A\left(  t-s\right)  }B_{2}u\left(  s-\tau_{2}\right)
\mathrm{d}s+\int_{t}^{t+\tau_{1}}\mathrm{e}^{A\left(  t-s\right)  }%
B_{1}u\left(  s-\tau_{1}\right)  \mathrm{d}s\nonumber \\
&  =\mathrm{e}^{-A\tau_{3}}\left(  \mathrm{e}^{A\tau_{3}}x\left(  t\right)
+\int_{t}^{t+\tau_{3}}\mathrm{e}^{A\left(  t+\tau_{3}-s\right)  }B_{3}u\left(
s-\tau_{3}\right)  \mathrm{d}s\right) \nonumber \\
&  \quad+\int_{t}^{t+\tau_{2}}\mathrm{e}^{A\left(  t-s\right)  }B_{2}u\left(
s-\tau_{2}\right)  \mathrm{d}s+\int_{t}^{t+\tau_{1}}\mathrm{e}^{A\left(
t-s\right)  }B_{1}u\left(  s-\tau_{1}\right)  \mathrm{d}s\nonumber \\
&  =x\left(  t\right)  +%
%TCIMACRO{\dsum \limits_{i=1}^{3}}%
%BeginExpansion
{\displaystyle \sum \limits_{i=1}^{3}}
%EndExpansion
\int_{t}^{t+\tau_{i}}\mathrm{e}^{A\left(  t-s\right)  }B_{i}u\left(
s-\tau_{i}\right)  \mathrm{d}s\nonumber \\
&  =x\left(  t\right)  +%
%TCIMACRO{\dsum \limits_{i=1}^{3}}%
%BeginExpansion
{\displaystyle \sum \limits_{i=1}^{3}}
%EndExpansion
\int_{t-\tau_{i}}^{t}\mathrm{e}^{A\left(  t-\tau_{i}-s\right)  }B_{i}u\left(
s\right)  \mathrm{d}s\nonumber \\
&  =z\left(  t\right)  , \label{eq14}%
\end{align}
which completes the proof.

\subsection*{A2. Proof of Lemma \ref{lm2}}

In view of systems (\ref{sys1}), (\ref{eq2}) and (\ref{eq3}), we can write
\begin{equation}
\left \{
\begin{array}
[c]{l}%
y_{3}\left(  t\right)  =x\left(  t+\tau_{3}\right)  -\int_{t}^{t+\tau_{3}%
}\mathrm{e}^{A\left(  t+\tau_{3}-s\right)  }\left(  B_{1}u\left(  s-\tau
_{1}\right)  +B_{2}u\left(  s-\tau_{2}\right)  \right)  \mathrm{d}s,\\
y_{2}\left(  t\right)  =y_{3}\left(  t+\tau_{2}\right)  -\int_{t}^{t+\tau_{2}%
}\mathrm{e}^{A\left(  t+\tau_{2}-s\right)  }\left(  B_{3}u\left(  s\right)
+\mathrm{e}^{A\tau_{3}}B_{1}u\left(  s-\tau_{1}\right)  \right)
\mathrm{d}s,\\
y_{1}\left(  t\right)  =y_{2}\left(  t+\tau_{1}\right)  -\int_{t}^{t+\tau_{1}%
}\mathrm{e}^{A\left(  t+\tau_{1}-s\right)  }\left(  \mathrm{e}^{A\tau_{2}%
}B_{3}+\mathrm{e}^{A\tau_{3}}B_{2}\right)  u\left(  s\right)  \mathrm{d}s,
\end{array}
\right.  \label{eq16}%
\end{equation}
from which it follows that%
\begin{align}
y\left(  t\right)   &  =\mathrm{e}^{-A\tau}y_{1}\left(  t\right) \nonumber \\
&  =\mathrm{e}^{-A\tau}y_{2}\left(  t+\tau_{1}\right)  -\int_{t}^{t+\tau_{1}%
}\mathrm{e}^{-A\tau}\mathrm{e}^{A\left(  t+\tau_{1}-s\right)  }\left(
\mathrm{e}^{A\tau_{2}}B_{3}+\mathrm{e}^{A\tau_{3}}B_{2}\right)  u\left(
s\right)  \mathrm{d}s\nonumber \\
&  =\mathrm{e}^{-A\tau}y_{3}\left(  t+\tau_{1}+\tau_{2}\right) \nonumber \\
&  \quad-\int_{t+\tau_{1}}^{t+\tau_{1}+\tau_{2}}\mathrm{e}^{-A\tau}%
\mathrm{e}^{A\left(  t+\tau_{1}+\tau_{2}-s\right)  }\left(  B_{3}u\left(
s\right)  +\mathrm{e}^{A\tau_{3}}B_{1}u\left(  s-\tau_{1}\right)  \right)
\mathrm{d}s\nonumber \\
&  \quad-\int_{t}^{t+\tau_{1}}\mathrm{e}^{-A\tau}\mathrm{e}^{A\left(
t+\tau_{1}-s\right)  }\left(  \mathrm{e}^{A\tau_{2}}B_{3}+\mathrm{e}%
^{A\tau_{3}}B_{2}\right)  u\left(  s\right)  \mathrm{d}s\nonumber \\
&  =\mathrm{e}^{-A\tau}x\left(  t+\tau_{1}+\tau_{2}+\tau_{3}\right)
\nonumber \\
&  \quad-\int_{t+\tau_{1}+\tau_{2}}^{t+\tau_{1}+\tau_{2}+\tau_{3}}%
\mathrm{e}^{-A\tau}\mathrm{e}^{A\left(  t+\tau_{1}+\tau_{2}+\tau_{3}-s\right)
}\left(  B_{1}u\left(  s-\tau_{1}\right)  +B_{2}u\left(  s-\tau_{2}\right)
\right)  \mathrm{d}s\nonumber \\
&  \quad-\int_{t+\tau_{1}}^{t+\tau_{1}+\tau_{2}}\mathrm{e}^{-A\tau}%
\mathrm{e}^{A\left(  t+\tau_{1}+\tau_{2}-s\right)  }\left(  B_{3}u\left(
s\right)  +\mathrm{e}^{A\tau_{3}}B_{1}u\left(  s-\tau_{1}\right)  \right)
\mathrm{d}s\nonumber \\
&  \quad-\int_{t}^{t+\tau_{1}}\mathrm{e}^{-A\tau}\mathrm{e}^{A\left(
t+\tau_{1}-s\right)  }\left(  \mathrm{e}^{A\tau_{2}}B_{3}+\mathrm{e}%
^{A\tau_{3}}B_{2}\right)  u\left(  s\right)  \mathrm{d}s\nonumber \\
&  =\mathrm{e}^{-A\tau}x\left(  t+\tau_{1}+\tau_{2}+\tau_{3}\right)
\nonumber \\
&  \quad-\int_{t+\tau_{1}+\tau_{2}}^{t+\tau_{1}+\tau_{2}+\tau_{3}}%
\mathrm{e}^{A\left(  t-s\right)  }B_{1}u\left(  s-\tau_{1}\right)
+\mathrm{e}^{A\left(  t-s\right)  }B_{2}u\left(  s-\tau_{2}\right)
\mathrm{d}s\nonumber \\
&  \quad-\int_{t+\tau_{1}}^{t+\tau_{1}+\tau_{2}}\mathrm{e}^{A\left(
t-\tau_{3}-s\right)  }B_{3}u\left(  s\right)  +\mathrm{e}^{A\left(
t-s\right)  }B_{1}u\left(  s-\tau_{1}\right)  \mathrm{d}s\nonumber \\
&  \quad-\int_{t}^{t+\tau_{1}}\mathrm{e}^{A\left(  t-\tau_{3}-s\right)  }%
B_{3}u\left(  s\right)  +\mathrm{e}^{A\left(  t-\tau_{2}-s\right)  }%
B_{2}u\left(  s\right)  \mathrm{d}s\nonumber \\
&  =\mathrm{e}^{-A\tau}x\left(  t+\tau_{1}+\tau_{2}+\tau_{3}\right)  -%
%TCIMACRO{\dsum \limits_{i=1}^{3}}%
%BeginExpansion
{\displaystyle \sum \limits_{i=1}^{3}}
%EndExpansion
\int_{t}^{t+\tau_{1}+\tau_{2}+\tau_{3}-\tau_{i}}\mathrm{e}^{A\left(
t-s-\tau_{i}\right)  }B_{i}u\left(  s\right)  \mathrm{d}s, \label{eq17}%
\end{align}
which completes the proof.

\bibliography{zb}
\bibliographystyle{unsrt}

\end{document}